\documentclass[11pt,a4paper]{amsart}

\usepackage{xy}

\usepackage{amsfonts}
\usepackage{amsthm}
\usepackage[T1]{fontenc}
\usepackage[utf8]{inputenc}
\usepackage{dsfont}

\xyoption{all}

\usepackage{color}
\usepackage{graphics}
\usepackage{amssymb}

\usepackage{enumitem}

 \theoremstyle{plain}
 \newtheorem{theorem}{Theorem}
 \newtheorem{proposition}{Proposition}[section]
 \newtheorem{lemma}[proposition]{Lemma}
 \newtheorem{corollary}{Corollary}

 \theoremstyle{definition}

 \newtheorem{remark}[proposition]{Remark}

\advance\textheight by 0.5cm
\advance\textwidth by 1.0cm
\advance\evensidemargin by -0.8cm
\advance\oddsidemargin by -0.8cm

\newcommand{\uplra}[1]{\stackrel{#1}{\lra}}
\newcommand{\upmapsto}[1]{\stackrel{#1}{\longmapsto}}

\newcommand{\podwzorem}[2]{\underbrace{#1}\limits_{#2}}
\newcommand{\nadwzorem}[2]{\overbrace{#1}\limits^{#2}}

\newcommand{\upcirc}[1]{\stackrel{#1}{\circ}}

\newcommand{\iso}{\cong}

\newcommand{\im}{\mathrm{Im}}

\newcommand{\rep}[4]{\left( \left\{#1_{#3} \right\}_{#3\in #2_0}, \left\{#1_{#4} \right\}_{#4\in #2_1}\right)}

\newcommand{\kwadracik}{\hfill\qed}

\newcommand{\modplus}[1]{\mathrm{mod}_+(#1)}
\newcommand{\modminus}[1]{\mathrm{mod}_-(#1)}
\newcommand{\modzero}[1]{\mathrm{mod}_0(#1)}
\newcommand{\cohplus}[1]{\mathrm{coh}_+(#1)}
\newcommand{\cohminus}[1]{\mathrm{coh}_-(#1)}

\newcommand{\id}{\mathds{1}}
\newcommand{\pc}[1]{ {\mathfrak{P}(#1)} }
\newcommand{\ih}[1]{ {\mathfrak{I}(#1)} }

\newcommand{\modf}[2]{#1+\im f_{#2}}

\newcommand{\Oo}{\mathcal{O}}

\newcommand{\XX}{\mathbb{X}}
\newcommand{\NN}{\mathbb{N}}
\newcommand{\LL}{\mathbb{L}}

\newcommand{\ZZ}{\mathbb{Z}}

\newcommand{\pp}{\underline{p}}
\newcommand{\lala}{\underline{\la}}

\newcommand{\al}{\alpha}
\newcommand{\La}{\Lambda}
\newcommand{\la}{\lambda}
\newcommand{\de}{\delta}

\newcommand{\tauX}{\tau_\XX}
\newcommand{\lra}{\longrightarrow}

\newcommand{\End}[2]{\mathrm{ End}_{#1}\left(#2\right)}
\newcommand{\EndX}[1]{\End{\XX}{#1}}

\newcommand{\Hom}[3]{\mathrm{ Hom}_{#1}\left(#2,#3\right)}
\newcommand{\HomX}[2]{\Hom{\XX}{#1}{#2}}

\newcommand{\extX}{\mathrm{ Ext}^1_{\XX}}
\newcommand{\Ext}[4]{\mathrm{ Ext}^{#1}_{#2}\left(#3,#4\right)}
\newcommand{\ExtX}[2]{\Ext{1}{\XX}{#1}{#2}}

\newcommand{\coh}[1]{\mathrm{coh}(#1)}
\newcommand{\vect}[1]{\mathrm{vect}(#1)}
\newcommand{\svect}[1]{\mathrm{\underline{vect}}(#1)}
\newcommand{\cohnull}[1]{\mathrm{coh}_0(#1)}
\renewcommand{\mod}[1]{\mathrm{mod}(#1)}

\newcommand{\rk}{\mathrm{rk}}
\newcommand{\Lp}{\mathbb{L}(\pp)}
\newcommand{\vx}{{\vec{x}}}
\newcommand{\vy}{\vec{y}}
\newcommand{\vc}{\vec{c}}

\newcommand{\vdom}{\vec{\delta}}

\newcommand{\vz}{{\vec{z}}}
\newcommand{\vw}{{\vec{\omega}}}

\newcommand{\extb}[2]{E_{#1}\langle#2\rangle}

\newcommand{\ramka}[1]{
	\begin{array}{|c|}
		\hline
		#1\\
		\hline
\end{array}}

\newcommand{\upramka}[2]{\stackrel{#1}{
		\begin{array}{|c|}
			\hline
			#2\\
			\hline
\end{array}}}

\newcommand{\blockmatrix}[6]{\begin{array}{|c|c|c|}
		\hline  #1& #2 & #3\\ \hline #4 & #5 & #6 \\ \hline
\end{array}}

\newcommand{\BlockMatrixDiagonalThree}[3]{{\normalsize\begin{array}{|c|c|c|}
			\hline
			#1&0&0\\
			\hline
			0&#2&0\\
			\hline
			0&0&#3\\
			\hline
\end{array}}}

\newcommand{\BlockMatrixDwoThree}[6]{{\normalsize\begin{array}{|c|c|c|}
			\hline
			#1&#2&#3\\
			\hline
			#4&#5&#6\\
			\hline
\end{array}}}

\newcommand{\BlockMatrixDwoDwo}[4]{{\normalsize\begin{array}{|c|c|}
			\hline
			#1&#2\\
			\hline
			#3&#4\\
			\hline
\end{array}}}

\begin{document}

\title[Extension modules]{Modules attached to extension bundles}

\author{Dawid Edmund Kędzierski}
\author{Hagen Meltzer}
\date{}

%\address{}

\address{Institut of Mathematics, Szczecin University, $70-451$
  Szczecin, Poland}

\email{dawid.kedzierski@usz.edu.pl,
hagen.meltzer@usz.edu.pl}

\subjclass[2010]{16G20, 14F05}

\keywords{exceptional module, canonical algebra, wild type, zero-one matrix problem, weighted projective line,  exceptional pair, extension bundles, Frobenius category, projective cover}

%\thanks{The second author was supported by the Polish Scientific Grant
 % KBN 1 P03A 007 27.}

%\dedicatory{In Memory}

\begin{abstract}
In  this article we study modules over wild canonical algebras
 which correspond to  extension bundles \cite{Kussin:Lenzing:Meltzer:2013adv} over weighted projective lines.
% In  \cite{Kedzierski:Meltzer:2020} it was proved that ''almost all'' exceptional modules over a wild canonical algebra
%  can be established by matrices with coefficients related to the  relations of the considered algebra.
%We prove that this statement holds true for all modules attached to  extension bundlescan be established by matrices with coefficients related to the  relations of the considered algebra.
We prove that  all modules attached to  extension bundles can be established by matrices with coefficients
 related to the  relations of the considered algebra.

Moreover,
we expand the concept  of extension bundles over weighted projective lines
with three weights to general weight type
 and establish similar results in this situation.
 Finally, we  present  a method to compute matrices for  all modules attached to extension bundles
   using  cokernels of maps between direct sums of line bundles.

%In this paper we will deal with algebraically closed fields.
\end{abstract}

\maketitle

\section{Introduction}

One of the problems of  representation theory of finite-dimensional algebras is the classification of indecomposable modules
over a given algebra.  Depending on the complexity of this issue we distinguish algebras
 of finite, tame and wild representation type.
In the case of wild algebras the structure of the module category is rich enough
that it is impossible to describe all indecomposable modules, however in this situation  sometimes it is
possible to describe subclasses of indecomposable modules.

 In this paper we study an important class of modules, namely the so called extension modules
 for wild canonical algebras.
Canonical algebras were introduced by C. M. Ringel in $1984$ \cite{Ringel:1984}, for a definition we
refer to Section \ref{notations}.

In the case of domestic canonical algebras D. Kussin and the second author \cite{Kussin:Meltzer:2007} described
 matrices for all indecomposable modules provided the characteristic of the field
 is different from $2$. In the case of characteristic $2$ matrices for the indecomposables were given in  \cite{Komoda:Meltzer:2008}.

 In the situation of tubular canonical algebras in \cite{Meltzer:2007}, extending methods of \cite{Ringel:1998}, it was
 shown that the exceptional modules can be exhibited by matrices having as coefficients only $0$, $1$ and $-1$
 in the cases $(3,3,3)$, $(2,4,4)$, $(2,3,6)$ and $0$, $1$,  $-1$, $\lambda$,  $\lambda -1$ in the case $(2,2,2,2)$,
 where $\la$ is the parameter appearing in the relations for the considered algebra.
 Based on this result in
 \cite {Dowbor:Meltzer:Mroz:2010} an algorithm and a computer program were developed to determine a
 description of all exceptional modules over tubular canonical algebras.
 Further in \cite{Dowbor:Meltzer:Mroz:2014bimodules} and \cite{Dowbor:Meltzer:Mroz:2014slope}
 the problem of homogeneous modules over tubular canonical modules was studied, in particular explicit
 matrices for modules of integral slope were given.

 For canonical algebras of wild type  it was proved by the authors in \cite{Kedzierski:Meltzer:2020} that ''almost all''
   exceptional modules can be described by matrices having coefficients $\la_i-\la_j$, where the $\la_i$
    are the parameters of the canonical algebra.

  In \cite{Geigle:Lenzing:1987} W. Geigle and H. Lenzing investigated weighted projective lines to give
   a geometric approach to canonical algebras.
 More precisely, they showed that the category  of coherent sheaves $\coh\XX$ over a weighted projective line $\XX$
 admits a tilting bundle $T$
 which induces an  equivalence of the bounded derived categories
$D^b(\coh\XX)\iso D^b(\mod\La)$,
where $\La= \End{\La}T$ is a canonical algebra.
%The category $\coh\XX$ is hereditary and it is worthwise
% to mention, that many problems concerning $\La$-modules can be translated to problems of coherent sheaves
 %in  $\coh\XX$ and can be solved easier there.

 In $2013$ D. Kussin, H. Lenzing and second author in \cite{Kussin:Lenzing:Meltzer:2013adv} introduced
 the concept of  extension bundles over  weighted projective lines with three weights,
 which is important in the study of nilpotent operators with invariant subspaces (see also \cite{Kussin:Lenzing:Meltzer:2013nil}).
 It was proved there, in particular,
 that each indecomposable vector bundle of rank two is exceptional and appears as the middle term of an exact sequence,
  where  the other terms are line bundles with good homological properties, see Section \ref{sec:three_weights}.

 The aim of this article  is to study modules attached to such extension bundles in the case of canonical algebras of
 wild type.
  Those modules are called \emph{extension modules}.
The paper contains the following results.

\begin{itemize}
	\item[\textbf{1.}] We  prove that all extension modules over a wild canonical algebra
$\La$ with three arms, can be described by matrices with entries $0$, $1$ and $-1$. This is an improvement for those modules
of the result from \cite{Kedzierski:Meltzer:2020}.
We will use the fact that the category of vector bundles $\vect\XX$ over a weighted projective
 line $\XX$ is a Frobenius category with the line bundles as the indecomposable projective-injective objects.
The main tool in the proof  is the fact that a vector bundle associated to a module
 has a line bundle, associated to a module, as a direct summand of its projective cover.

	\item[\textbf{2.}] We extend the concept of extension bundles from \cite{Kussin:Lenzing:Meltzer:2013adv}
to the case of an arbitrary number of weights.
If this number  is greater than $3$, then not every indecomposable vector bundle of rank two is exceptional.
We  present a useful characterization of  exceptional modules of rank two
 as  extension bundles with  data $(L,\vx)$, where $L$ is a line bundle and $\vx$ is an element of the grading
 group of a specific form.
 We also  establish the projective covers and the injective hulls of those bundles.
	
	\item[\textbf{3.}]  We  show that all extension modules for a wild canonical algebra with
 an arbitrary  number of arms can be established by matrices with coefficients $0$ , $ \la_i$, $- \la_i$
 where the $\la_i$ are the parameters of the canonical algebra.
	
\item[\textbf{4.}] We  compute matrix representation for each extension module over a canonical algebra
of arbitrary type. Since the method using Schofield induction for exceptional modules (see
\cite{Ringel:1998}, \cite{Kedzierski:Meltzer})  is not constructive we can not proceed as in the case
for  tubular canonical algebras \cite{Meltzer:2007}.
Therefore here we present another idea.
We show that  each extension module appears as a cokernel of a map between  direct sums of  line bundles
and we describe a method to calculate matrices for these cokernels.

%	\item[\textbf{4.}] We will compute matrix representation for each extension module. In one way we can use Ringe's maps %\cite{Ringel:1998,Meltzer:2007} and the results from \cite{Kedzierski:Meltzer:2020} to determined representations for each extension modules. But %this is a hard way to deal with it. We will present better idea, by using quite simple a cokernel construction. The stepping stone is to prove that %each extension bundle appear as a cokernel of the map between the direct sums of the line bundles, and in the case of an extension module we will %deal with the map between the direct sums of the rank one modules.

  %Lastly, we will justify the correctness of the received representations.
\end{itemize}

%-------------------------------------------------------------------------
%-------------------------------------------------------------------------
%										Podstawowe pojęcia
%-------------------------------------------------------------------------
%-------------------------------------------------------------------------

\section{Notations and basic concepts}\label{notations}

Let $k$ be an algebraically closed field.
We recall the concept of a weighted projective line in the sense of W.  Geigle and H- Lenzing \cite{Geigle:Lenzing:1987}.
Let $\LL=\LL(\pp)$ be the rank one abelian group with generators
$\vx_1,\dots,\vx_t$
and relations $p_1\vx_1=\cdots p_t\vx_t:=\vc$, where the  $p_i$ are integers greater than or equal to $2$.
 These numbers are called \emph{weights}.
  The element
$\vc$ is called the \emph{canonical element}.
 Recall that $\LL$ is an ordered group with $\LL_+=\sum_{i=1}^t\NN \vx_i$ as its set of non-negative elements.
 Moreover, each element $\vy$ of $\LL$ can be written in \emph{normal form} $\vy=a\vc+\sum_{i=1}^t a_i\vx_i$
 with $a\in\ZZ$ and $0\leq a_i<p_i$.
The polynomial algebra $k[x_1,\dots,x_t]$ is $\LL-$graded,
where the degree of $x_i$ is $\vx_i$.
Since the polynomials $f_i=x^{p_i}_i-x_1^{p_1}-\la_ix_2^{p_2}$ for $i=3,\dots t$ are homogeneous, the quotient algebra $S=k[x_1,\dots,x_t]/\langle f_i\mid i=3,\dots,t\rangle$ is also $\LL-$graded.
Here the $\la_i$ are pairwise distinct non-zero elements of $k$, they are called the \emph{parameters}.
 A \emph{weighted projective line} $\XX$ is the projective spectrum of the $\LL-$graded algebra $S$.
 Therefore $\XX$ depends on a weight sequence $\pp=(p_1,\dots, p_t)$ and a sequence of parameters
  $\lala= (\la_3,\dots,\la_t)$. We can assume that  $\la_3=1$.
  The category of coherent sheaves over $\XX$ will be denoted by $\coh\XX$.
  Each indecomposable sheaf in $\coh\XX$ is a locally free sheaf, called a \emph{vector bundle},
   or a \emph{sheaf of finite length}. Denote by $\vect\XX$ (resp. $\cohnull\XX$)
    the subcategory of  $\coh\XX$ consisting of all vector bundles (resp. finite length sheaves) on $\XX$.

The category $\coh\XX$ is a $\mathrm{Hom}-$finite, abelian $k-$category. Moreover, it is hereditary that is $\Ext i\XX --=0$ for $i\geq 2$ and  has Serre duality in the form $\ExtX FG\iso D\HomX G{\tauX F}$, where the Auslander-Reiten translation $\tauX$ is given by the shift $F\mapsto F(\vw)$, where $\vw:=(t-2)\vc-\sum_{i=1}^t\vx_i$ denotes the \emph{dualizing element}.
It is well known that each line bundle has the form $\Oo(\vx)$
where $\Oo$ is  the structure sheaf of $\XX$ and where $\vx\in \LL$.
Furthermore we have isomorphisms $\HomX{\Oo(\vx)}{\Oo(\vy)}\iso S_{\vy-\vx}$, where
$ S_{\vz}$ denotes the
grading component of $S$ associated to $\vz\in \LL$.

One of the main results proved in \cite{Geigle:Lenzing:1987} is the fact that in $\coh\XX$
there is a tilting object, which is a direct sum of line bundles  $T=\bigoplus_{0\leq \vx\leq \vc}\Oo(\vx)$,
such that the right derived functor of the functor  $\HomX T-$
 induces an equivalence of  bounded derived category $\mathcal{D}^b(\coh\XX)\stackrel{\iso}{\lra}\mathcal{D}^b(\mod\La)$,
where $\La=\EndX T$ is a canonical algebra,
  called the  \emph{canonical algebra} associated to the weighted projective line $\XX$.

Originally, canonical algebras $\La$ were introduced by C. M. Ringel \cite{Ringel:1984} as path algebras of quivers $Q$:
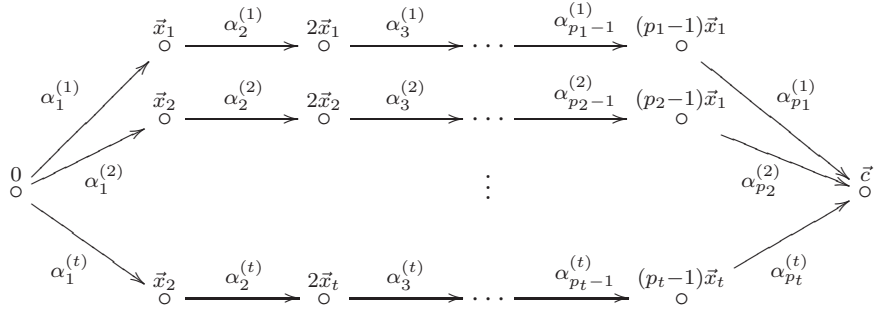
$$\xymatrix @C +1.5pc @R-1.5pc{&\upcirc{\vx_1}\ar[r]^{\al_2^{(1)}}& \upcirc{2\vx_1}\ar[r]^{\al_3^{(1)}}&\cdots \ar[r]^{\al_{p_1-1}^{(1)}}&\upcirc{(p_1-1)\vx_1} \ar[rdd]^{\al_{p_1}^{(1)}}&\\
	&\upcirc{\vx_2}\ar[r]^{\al_2^{(2)}}& \upcirc{2\vx_2}\ar[r]^{\al_3^{(2)}}&\cdots \ar[r]^{\al_{p_2-1}^{(2)}}&\upcirc{(p_2-1)\vx_1}\ar[rd]_{\al_{p_2}^{(2)}}&\\
	\upcirc{{0}}\ar[ruu]^{\al_1^{(1)}} \ar[ru]_{\al_1^{(2)}} \ar[rdd]_{\al_1^{(t)}}  &&&\vdots&&\upcirc\vc\\
	&& && &\\
	&\upcirc{\vx_{2}}\ar[r]^{\al_2^{(t)}}& \upcirc{2\vx_{t}}\ar[r]^{\al_3^{(t)}}&\cdots \ar[r]^{\al_{p_{t}-1}^{({t})}}&\upcirc{(p_{t}-1)\vx_{t}}\ar[ruu]_{\al_{p_{t}}^{({t})}}&\\
}$$
with \emph{canonical relations}
$$\al_{p_i}^{(i)}\dots\al_{2}^{(i)}\al_1^{(i)}=\al_{p_1}^{(1)}\dots \al_{2}^{(1)}\al_1^{(1)}+\lambda_i\al_{p_2}^{(2)}\dots \al_{2}^{(2)}\al_1^{(2)}\quad \textnormal{for}\quad i=3,4,\dots,t,$$
where the $\la_i$ are parameters from $\lala$ and $p_i$ are weights from $\pp$ as before.
We call $t$ the number of arms of $\La$.
Concerning the complexity of the module category over $\La$ there are three types
 of canonical algebras, domestic, tubular and wild ones.
 Recall that  $\La$ is of domestic (respectively tubular, wild) type if
  the Euler characteristic $\chi_\La= (2-t)+\sum_{i=1}^t1/{p_i}$ is positive (respectively zero, negative).

Denote by $Q_0$ the set of vertices and by $Q_1$ the set of arrows of the quiver $Q$. Then each finitely
 generated right module over $\La$ is given by finite dimensional vector spaces $M_i$ for each vertex $i$ of $Q_0$ and by linear maps $M_\al:M_j\rightarrow M_i$ for each arrows $\al:i\rightarrow j$ of
$Q_1$ such that the canonical relations are satisfied.
 We will usually identify  linear maps with matrices.
 The category of  finite generated right  modules we denote by $\mod\La$.

For coherent sheaves there are well known invariants the \emph{rank}, the \emph{degree} and the  \emph{determinant},
which are given by linear forms on the Grothendieck group $\rk, \deg : K_0(\XX)\lra\ZZ$ and $\det:K_0(\XX)\lra \LL(\pp)$.
%again called rank, degree and determinant.
Since $  K_0(\XX)\simeq  K_0(\La)  $ we have also the concept of the rank, the degree and the determinants for $\La$-modules
In particular the rank of a $\La$-module is defined by the formula $\rk M:= \dim_k M_{0}-\dim_k M_{\vc}$.
We denote by $\modplus\La$ (respectively $\modminus\La$, $\modzero\La$) the full subcategory consisting of all
 $\La-$modules, whose indecomposable summands in the decomposition into a direct sum have positive
  (respectively negative or zero) rank. Further, by $\cohplus\XX$ (resp. $\cohminus\XX$) we denote the full subcategory
  of all vector bundles $X$ on $\XX$, such that the functor $\ExtX TX=0$ (resp. $\HomX TX=0$).
  Under the equivalence $\mathcal{D}^b(\coh\XX)\stackrel{\iso}{\lra}\mathcal{D}^b(\mod\La)$
\begin{itemize}
	\item $\cohplus\XX$ corresponds to $\modplus\La$ by means of $E\mapsto \HomX TE$,
	\item $\cohnull\XX$ corresponds to $\modzero\La$ by means of $E\mapsto \HomX TE$,
	\item $\cohminus\XX[1]$ corresponds to $\modminus\La$ by means of
$E[1]\mapsto \ExtX TE$, where $[1]$ denotes the suspension functor of the triangulated category \\
$D^b(\coh\XX)$.
\end{itemize}

We say in these cases that the module $\HomX TE$ (respectively $\ExtX TE$  \emph {is attached} to $E$.
For simplicity we will often identify a sheaf $E$ in $\cohplus\XX$ or $\cohnull\XX$ with the corresponded $\La-$module $\HomX TE$.

\begin{remark}\label{minusplus}
Recall from \cite[Theorem 9.1.1]{Drozd:Kirichenko:1994} that the standard duality $\Hom{\La}-k$ defines an equivalence of the categories $\mod\La$ and $\mod{\La^{op}}$. Under this equivalence $\modminus\La$ corresponds to $\modplus{\La^{op}}$. Since
 $\La\simeq\La^{op}$ in many considerations it is sufficient to consider  $\La-$modules of positive rank and of rank zero.
 In particular,  modules of negative rank
can be obtained from those of positive rank
by reversing the arrows and transposing the matrices.
% of the representations from $\modplus\La$.
\end{remark}

Recall that a coherent sheaf $E$ over $\XX$ is called \emph{exceptional} if
 $\EndX E=k$ and $\ExtX EE=0$.
  A pair $(X,Y)$ in $\coh\XX$ is called
an \emph{exceptional pair}  if $X$, $Y$ are exceptional and $\HomX YX=0=\ExtX YX$.
Furthermore, an exceptional pair is \emph{orthogonal} if addition $\HomX XY=0$.
 Finally, a $\La-$module $M$ is called \emph{exceptional} if $\End{\La}M =k$ and $\Ext i{\La}MM=0$ for $i\geq 1$.

\section{Extension bundles for weighted projective lines with three weights} \label{sec:three_weights}

Let $\XX$ be a weighted projective line of a triple type $(p_1,p_2,p_3)$. The concept of extension bundles was introduced
%by D. Kussin, H. Lenzing and the second author
in \cite{Kussin:Lenzing:Meltzer:2013adv} in the study of  stable vector bundle categories.
In particular it was shown that stable vector bundle categories of weighted projective lines of triple type
admit tilting objects, being direct sums of extension bundles, such that their endomorphism algebras form cuboids.

From \cite[Theorem 4.2.]{Kussin:Lenzing:Meltzer:2013adv} each indecomposable vector bundle $E$ can be obtained as the middle term of a non-split exact sequence
$$\eta_{L,\vx}:\quad 0\lra L(\vw)\lra E \lra L(\vx)\lra 0$$
for  some line bundle $L$ and some element $\vx$ of $\LL$ such that $0\leq \vx\leq \vdom$, where $\vdom:=\vc+2\vw$ is the \emph{dominant element}.
Because in this case the vector space $\ExtX {L(\vx)}{L(\vw)}$ is one-dimensional the bundle $E$ is uniquely determined up to isomorphism.
It is called the \emph{extension bundle} given by the pair $(L,\vx)$.
It is easy to check, that the pair $(L(\vx), L(\vw))$ is exceptional and orthogonal.
 Therefore, if $L(\vx)$ and $L(\vw)$ are $\La-$modules, both in $\modplus\La$  or in $\modminus\La$,
  then they can be described by  $0,\pm\la_i$ matrices, as rank one modules (see \cite{Meltzer:2007})
   and it follows that $E$ also can be described by matrices
  with the same coefficients (see Proposition 7.1 and the remark after its proof in \cite{Kedzierski:Meltzer:2020}).

We recall that the category $\vect\XX$ of vector bundles over $\XX$ is a Frobenius category
 such that the indecomposable projective-injective objects are exactly the line bundles
 \cite[Def. 3.1]{Kussin:Lenzing:Meltzer:2013adv}. Moreover each vector bundle has a projective cover and an injective hull.

\begin{lemma}
	\label{lem:projektywny_skladnik_jest_modulem}
Let $E$ be a non-zero vector bundle on  a weighted projective line
$\XX$ of  type $(p_1,p_2,p_3)$ with  projective cover $\pc{E}$.
 If $\ExtX TE=0$ for the canonical bundle $T$, then there is a line bundle
  $L$ in the decomposition $\pc E$ ínto a direct sum of line bundles, such that $\ExtX TL=0$.
\end{lemma}

\begin{proof}
Assume that $\ExtX TE=0$. Then $\HomX TE\neq 0$ because $T$ is a tilting bundle and $E$ is non-zero.
Therefore there is an element $\vx$, such that $0\leq \vx\leq \vc$ and  $\HomX {\Oo(\vx)}E\neq 0$.
Each non-zero morphism $f:\Oo(\vx)\lra E$ factors through $\pi_E:\pc{E}\lra E$, so there is morphism $0\neq\alpha:\Oo(\vx)\lra \pc{E}$ such that $\pi_E\circ\alpha=f$. Hence there is a direct summand $L=\Oo(\vy)$  of $\pc E$ such that $\HomX{\Oo(\vx)} {\Oo(\vy)}\neq 0$.

 We will show that $L$ has the desired property.
Writing $\vy$ in normal form $\vy=\alpha\vc+\alpha_1\vx_1+\alpha_2\vx_2 +\alpha_3\vx_3$, with  $0\leq\alpha_i\leq p_i-1$, we obtain $\alpha\geq 0$ and
$$(\triangle)\quad\vc+\vw-\vy=(2-\alpha)\vc-\sum_{i=1}^3(\alpha_i+1)\vx_i<0.$$

Assume that $\ExtX T{\Oo(\vy)}\neq 0$ and let $\vz$ satisfy $\ExtX {\Oo(\vz)}{\Oo(\vy)}\neq 0$ and $0\leq \vz\leq\vc$. Using Serre duality we obtain $\vz+\vw-\vy\geq 0$. Therefore
$$\begin{array}{ccccc}
\vc+\vw-\vy=&\underbrace{(\vc-\vz)}&+&\underbrace{(\vz+\vw-\vy)}&\geq 0,\\
&\geq 0&&\geq 0&\\
\end{array}$$
a contradiction with $(\triangle)$. Thus $\ExtX T{\Oo(\vy)}= 0$
\end{proof}

In the following lemma we prove that each extension bundle defined by a short exact sequence
$\eta_{L,\vx}$
 appears in addition as an extension bundle  for three different pairs $(L,\vx)$.

\begin{lemma} \label{lem:przedstawienie_wiazki_rozszerzen_na_4_sposoby}
Let $\XX$ be a weighted projective line of  type $(p_1,p_2,p_3)$ and let
$E$ be an extension bundle given by a pair $(L,\vx)$, where $\vx=l_1\vx_1+l_2\vx_2+l_3\vx_3$.
Then $E$ is also an extension bundle determined by the following pairs:
$$\left(L(\vx-(1+l_i)\vx_i)(-\vw),\quad l_i\vx_i+\sum_{j\neq i}(p_j-l_j-2)\vx_j \right) \quad\textnormal{for}\quad i=1,2,3,$$
where $L(\vx-(1+l_i)\vx_i)$ are  direct summands of $\pc E$. In particular,
for each $i\in\{1,2,3\}$, there is an exact short sequence 	
$$0\lra L(\vx-(1+l_i)\vx_i)\lra E\lra L((1+l_i)\vx_i+\vw)\lra 0.$$
\end{lemma}

\begin{proof} The projective cover of $E$ has the  form  $\pc{E}=L(\vw)\oplus \bigoplus_{i=1}^3 L(\vx-(1+l_i)\vx_i)$, \cite[Theorem 4.6.]{Kussin:Lenzing:Meltzer:2013adv}.
Then there are exact sequences
$$\eta_i:\quad 0\lra L(\vx-(1+l_i)\vx_i)\lra E\lra \widehat{L}_i\lra 0,\quad \text{for}\quad i=1,2,3,$$
where from \cite[Proposition 3.8]{Kussin:Lenzing:Meltzer:2013adv}
the sheaf  $\widehat{L}_i$ is a line bundle, for  $i=1,2,3$.
 By comparison of the determinants we obtain that $\widehat{L}_i=L((1+l_i)\vx_i+\vw)$ is a direct summand of the injective hull of $E$. Therefore the sequence $\eta_i$ can be presented as follows
$$0\lra L(\vx-(1+l_i)\vx_i-\vw)(\vw)\lra E\lra L(\vx-(1+l_i)\vx_i-\vw)(\vy_i)\lra 0,$$
where $\vy_i=(1+l_i)\vx_i+2\vw-\vx=l_i\vx_i+\sum_{j\neq i}(p_j-l_j-2)\vx_j.$
Since $0\leq l_j\leq p_j-2$ we have $0\leq p_j-l_j-2\leq p_j-2$, and consequently $0\leq \vy_i\leq \vdom$.
\end{proof}

\begin{theorem} \label{thm:three:weights}
Let $\La$ be a canonical algebra with three arms. Then each indecomposable $\La-$module of rank two
 can be described by matrices having coefficients $0$, $ 1$, $-1$.
\end{theorem}

\begin{proof}
Let $M$ be a $\La-$module of rank $2$, attached to an indecomposable vector bundle $E$
over the weighted projective line $\XX$ associated to $\La$. Then $M$ is in $\modplus\La$.
We will show that there is an exact sequence
$$\eta: \quad 0\lra A\lra E\lra B\lra 0,$$ where $A,B  \in \modplus\La$  and $(B,A)$ is an orthogonal exceptional pair
in $\coh \XX$.
Then the result follows from \cite[Proposition 7.1]{Kedzierski:Meltzer:2020}.

From \cite[Theorem 4.2]{Kussin:Lenzing:Meltzer:2013adv} the vector bundle $E$
appears as an extension bundle given by a pair $(L,\vx)$, this means that there is a
short exact sequence
 $$\eta_{L,\vx}:\quad0\lra L(\vw)\lra E\lra L(\vx)\lra 0,$$
where $(L(\vx),L(\vw))$ is an orthogonal exceptional pair.
Since the vector bundle $E$ is attached to the $\La-$module $M$ we have $\ExtX TE=0$.
Applying the functor $\HomX T-$ to the sequence $\eta_{L,\vx}$ we obtain that
$\extX(T,L(\vx))=0$, thus
 $L(\vx)$ is  in $\modplus\La$. If in addition $\extX(T,L(\vw))=0$ we are done.
 Otherwise we will replace the exact sequence $\eta_{L,\vx}$ by another one.

To do so we
recall that $L(\vw)$ is a direct summand of the projective cover $\pc E$ and from
Lemma \ref{lem:projektywny_skladnik_jest_modulem} there is a line bundle $\widehat{L}$,
which is a direct summand of  $\pc{E}$ and  satisfies  $\extX(T,\widehat{L})=0$
thus $\widehat{L}$ is in $\cohplus\XX=\modplus\La$.
From \cite[Theorem 4.6]{Kussin:Lenzing:Meltzer:2013adv}
the line bundle $\widehat{L}$ has the form
   $L(\vx-(1+l_i)\vx_i)$ for some $i\in\{1,2,3\}$.
    Thus using Lemma \ref{lem:przedstawienie_wiazki_rozszerzen_na_4_sposoby} we get  an exact sequence
$$0\lra L(\vx-(1+l_i)\vx_i)\lra E\lra L((1+l_i)\vx_i+\vw)\lra 0.$$

 Applying the functor $\HomX T-$ to the sequence above we conclude
  that $\extX(E,L((1+l_i)\vx_i+\vw)= 0$
 and therefore $L((1+l_i)\vx_i+\vw)$ is in $\modplus\La$. Moreover it is easily checked
 that $\big(L(\vx-(1+l_i)\vx_i), L((1+l_i)\vx_i+\vw)\big)$ form an orthogonal exceptional pair in
 $\coh\XX$. Thus we get an exact sequence $\eta$ of the desired form and the theorem is proved.
\end{proof}

\begin{remark}
Using Remark  \ref{minusplus}
we get the same result for exceptional modules of rank $-2$ from $\modminus\La$.
\end{remark}

%-------------------------------------------------------------------------
%-------------------------------------------------------------------------
%				Extension bundles adapted to $t$ numbers of weights
%-------------------------------------------------------------------------
%-------------------------------------------------------------------------

%\section{Extension bundles adapted to $t$ numbers of weights}
\section{Extension bundles in the case of $t$ numbers of weights}

In this section we will deal with
 a weighted projective line $\XX$ of the type $(p_1,\dots,p_t)$ where $t$ is greater than of equal to $3$.

\begin{theorem}\label{prop:extension_bundle}
Let $\XX$ be a weighted projective line of a type $(p_1,\dots,p_t)$.
Then each indecomposable vector bundle of rank two occurs as the middle term of a non-split exact sequence
$$\eta:\quad 0 \lra  L(\vw) \stackrel{i}{\lra}   E  \stackrel{\pi}{\lra}  L(\vx)\lra 0,$$
where $0\leq \vx\leq \vdom:=2\vw+\vc=(t-3)\vc+\sum_{i=1}^t(p_i-2)\vx_i$.
Moreover the following conditions are equivalent:

\begin{enumerate}[label=\upshape(\roman*), leftmargin=*, widest=iii]
\item  The vector bundle $E$ is exceptional.
\item The pair $(L(\vx),L(\vw))$ is an orthogonal exceptional pair with $\ExtX{L(\vx)}{L(\vw)}=k$.
\item $\vx=\sum_{i=1}^tl_i\vx_i$ with $0\leq l_i\leq p_i-1$ and there are exactly $t-3$ numbers $l_i$ equal to $p_i-1$.
\item
$\ExtX {L(\vw)}{L(\vx)}=0$,  $\HomX E{L(\vw)}=0=\ExtX E{L(\vw)}$ and
\\$\HomX {L(\vx)}E=0=\ExtX {L(\vx)}E$.
\end{enumerate}

\end{theorem}
\begin{proof}
The proof of the existence of the sequence $\eta$ is almost the same as in the
case of a triple weight type,  we refer the reader to \cite[Theorem 4.2]{Kussin:Lenzing:Meltzer:2013adv}.

$(i)\Rightarrow (ii)$. Assume that $E$ is exceptional and $0\leq \vx\leq \vdom$. Then $\vx$ can be written in normal form $\vx=l\vc+\sum_{i=1}^tl_i\vx_i$, with $l\geq 0$ and $0\leq l_i\leq p_i-1$ and $\vx\leq \vdom$.
Because $\vx-\vw\leq \vdom-\vw=\vw-\vc<0$ we have $\HomX {L(\vw)}{L(\vx)}\iso S_{\vx-\vw}=0$.
Similarly, the vector space $\HomX {L(\vx)}{L(\vw)}$  also vanishes. From  Serre duality we get
$$\ExtX {L(\vx)}{L(\vw)}\iso \HomX {L(\vw)}{L(\vx+\vw)}\iso S_{\vx}\iso k^{l+1}.$$
Now, we have $[E]=[L(\vw)]+[L(\vx)]$ in the Grothendieck group $K_0(\XX)$ and applying the
 Euler form $\langle-,-\rangle_\XX:K_0(\XX)\times K_0(\XX)\lra \ZZ$ we obtain
\begin{equation}\nonumber
\begin{split}
1&=\langle[E],[E]\rangle_\XX= \langle[L(\vw)]+[L(\vx)],[L(\vw)]+[L(\vx)]\rangle_\XX\\
&= 2+ \dim_k\HomX{L(\vw)}{L(\vx)}-\dim_k\ExtX{L(\vw)}{L(\vx)}\\
&+\dim_k\HomX{L(\vx)}{L(\vw)}-\dim_k\ExtX{L(\vx)}{L(\vw)}\\
&=2-(l+1)-\dim_k\ExtX{L(\vw)}{L(\vx)}.
\end{split}
\end{equation}
Therefore $-l=\dim_k\ExtX{L(\vw)}{L(\vx)}$ and it follows that $l\geq 0$.
Consequently $\dim_k\ExtX{L(\vw)}{L(\vx)}=0$ and $\ExtX {L(\vx)}{L(\vw)}=k$.

$(ii)\Rightarrow (iii)$. Assume that $(L(\vx), L(\vw))$ is an orthogonal exceptional pair,
 such that $\ExtX {L(\vx)}{L(\vw)}\iso k$. The element $\vx$ can be written in  normal form
  $\vx=l\vc+\sum_{i=1}^tl_i\vx_i$, with $l\geq 0$ and $0\leq l_i\leq p_i-1$. From  Serre duality we obtain that
$\ExtX {L(\vx)}{L(\vw)}\iso k^{l+1}$. Hence $l=0$.
 Moreover
$$0=\dim_k\ExtX{L(\vw)}{L(\vx)}\iso D\HomX{L(\vx)}{L(2\vw)}\iso S_{2\vw-\vx},$$
where $2\vw-\vx=(t-4)\vc+\sum_{i=1}^{t}(p_i-2-l_i)\vx_i<0$. Therefore
at least $t-3$ numbers $l_i$ have to be equal to $p_i-1$. Furthermore,
because $\vx\leq\vdom$ at most $t-3$ numbers $l_i$ can be equal to $p_i-1$.
This implies that exactly $t-3$ numbers $l_i$ are equal to $p_i-1$.

$(iii)\Rightarrow (iv)$. Assume that the element
$\vx$ has  normal form $\sum_{i=1}^tl_i\vx_i$ with $0\leq l_i\leq p_i-1$ and there
 are exactly $t-3$ numbers $l_i$ equal $p_i-1$. Therefore by Serre duality we get
\begin{align*}
\ExtX{L(\vw)}{L(\vx)}&\iso D\HomX{L(\vx)}{L(2\vw)}\iso S_{2\vw-\vx}=0,\\
\ExtX{L(\vx)}{L(\vw)}&\iso D\HomX{L(\vw)}{L(\vx+\vw)}\iso S_{\vx}\iso k.
\end{align*}
Applying the functor $\HomX{L(\vx)}{-}$ to $\eta$ we get an exact sequence
\begin{equation}\nonumber
\begin{split}
0\lra &\podwzorem{\HomX{L(\vx)}{L(\vw)}}{=\ 0}\lra \HomX{L(\vx)}E\lra \podwzorem{\HomX{L(\vx)}{L(\vx)}}{\iso\ k} \lra\\
\uplra{\delta}  &\podwzorem{\ExtX{L(\vx)}{L(\vw)}}{\iso\ k}\lra \ExtX{L(\vx)}E\lra \podwzorem{\ExtX{L(\vx)}{L(\vx)}}{=\ 0}\lra 0.
\end{split}
\end{equation}
Now, $\delta(\id_{L(\vx)})=\eta\neq 0$ because $\eta$ does not split and consequently $\delta$ is isomorphism. Therefore
$$\HomX{L(\vx)}F=0=\ExtX{L(\vx)}F.$$

The long exact  sequence $\HomX{\eta}{L(\vw)}$ has the form
\begin{equation}\nonumber
\begin{split}
0\lra& \podwzorem{\HomX{L(\vx)}{L(\vw)}}{=\ 0,\text{ because }\vw-\vx<0 } \uplra{-\circ\pi}\HomX E{L(\vw)}\uplra{-\circ i} \podwzorem{\HomX{L(\vw)}{L(\vw)}}{\iso\ k} \lra \\
\lra& \podwzorem{\ExtX{L(\vx)}{L(\vw)}}{\iso\ k}\lra\ExtX E{L(\vw)}\lra \podwzorem{\ExtX{L(\vw)}{L(\vw)}}{=\ 0}\lra 0.\\
\end{split}
\end{equation}
Let  $u:E\lra L(\vw)$ be a non-zero morphism.  Then $u\circ i:L(\vw)\lra L(\vw)$ is the
zero map. Indeed, if $u\circ i$ is non-zero, it is
an isomorphism and so $\eta$ splits which is impossible. Hence
$u\in\ker(-\circ i)=\im(-\circ \pi)=0$, because $\HomX{L(\vx)}{L(\vw)}=0$.
Then  $\HomX{E}{L(\vw)}=0$ and by comparing dimensions in the sequence $\HomX{\eta}{L(\vw)}$ we obtain
that  $\ExtX{E}{L(\vw)}=0$.

$(iv)\Rightarrow (i)$.  Assume that the vector spaces $\ExtX {L(\vw)}{L(\vx)}$, $\HomX E{L(\vw)}$, $\ExtX E{L(\vw)}$,
$\HomX {L(\vx)}E$ and $\ExtX {L(\vx)}E$ vanish.
We apply the functor $\HomX {L(\vw)}-$ to $\eta$ and obtain a long exact sequence
\begin{equation}\nonumber
\begin{split}
0\lra& \podwzorem{\HomX{L(\vw)}{L(\vw)}}{\iso k} \lra\HomX {L(\vw)}E\lra \podwzorem{\HomX{L(\vw)}{L(\vx)}}{=0} \lra \\
\lra& \podwzorem{\ExtX{L(\vw)}{L(\vw)}}{=0}\lra\ExtX {L(\vw)}E\lra \podwzorem{\ExtX{L(\vw)}{L(\vx)}}{=0}\lra 0.\\
\end{split}
\end{equation}
Hence $\HomX{L(\vw)}E\iso k$ and $\ExtX{L(\vw)}E=0$.
Finally, we  apply the functor $\HomX -E$ to  $\eta$ and obtain a long exact sequence
\begin{equation}\nonumber
\begin{split}
0\lra& \podwzorem{\HomX{L(\vx)}{E}}{=0} \lra\HomX EE\lra \podwzorem{\HomX{L(\vw)}{E}}{\iso k} \lra \\
\lra& \podwzorem{\ExtX{L(\vx)}E}{=0}\lra\ExtX EE\lra \podwzorem{\ExtX {L(\vw)}E}{=0}\lra 0.\\
\end{split}
\end{equation}
Therefore $\HomX EE=k$ and $\ExtX EE=0,$ so the vector bundle $E$ is exceptional.
\end{proof}

For an exceptional bundle $E$ the non-split sequence $\eta$ uniquely determines $E$,
and in this case we will say that the \emph{extension bundle} is given by the pair $(L,\vx)$
and we will denote it by $E_L\langle\vx\rangle$.

\begin{theorem}\label{thm:projective_covers_t-weight}
	Let $E$ be an indecomposable vector bundle over $\XX$ such that there is a short exact sequence
	$$\eta:\quad 0\lra L(\vw)\stackrel{i}{\lra} E \stackrel{\pi}{\lra} L(\vx)\lra 0,$$
	where $\vx=\sum\limits_{i=1}^t l_i\vx_i$, with $0\leq l_i\leq p_i-1$ and $0\leq \vx\leq \vdom$.
	Moreover, let $I=\{i\mid l_i\neq p_i-1\}$. Then
\begin{equation}\nonumber
\begin{split}
\pc {E_L\langle\vx\rangle}&= L(\vw)\oplus\bigoplus_{j\in I}L(\vx-(1+l_j)\vx_j)\\
\ih {E_L\langle\vx\rangle}&= L(\vx)\oplus\bigoplus_{j\in I}L((1+l_j)\vx_j+\vw)
\end{split}
\end{equation}
Further, the line bundle summands of $\pc {E_L\langle\vx\rangle}$ (resp. $\ih {E_L\langle\vx\rangle}$)
are mutually $\mathrm{Hom}_\XX-$orthogonal.
\end{theorem}

\begin{proof}
Observe that the condition $0\leq \vx\leq \vdom$ implies that $\#I\geq 3$.	
%The $\mathrm{Hom}_\XX-$orthogonality is easy to check.
We will consider the case of injective hulls, the arguments for projective covers are dual.

From  Serre duality we obtain that $\ExtX {L(\vx)}{L(\vw+(1+l_j)\vx_j)}=0$ for $j\in I$. Hence, applying the functor $\HomX-{L(\vw+(1+l_j)\vx_j)}$ for $j\in I$ to $\eta$ we see
that there are morphisms $\widetilde{x_j^{l_j+1}}: E\lra L(\vw+(1+l_j)\vx_j)$
such that $\widetilde{x_j^{l_j+1}}\circ i=x_j^{l_j+1}$ where $x_j^{l_j+1}:L(\vw)\lra L(\vw+(1+l_j)\vx_i)$.
We will show that  $j_E=\left(\pi, \left(\widetilde{x_j^{l_j+1}}\right)_{j\in I}\right):E\lra L(\vx)\oplus\bigoplus_{j\in I}L(\vw+(1+l_j)\vx_j)$
 is an injective hull of the bundle $E$.
For this we
 will prove that each morphism $E \lra  L'$ where $L'$ is a line bundle, factors through $L(\vx)\oplus\bigoplus_{j\in I}L(\vw+(1+l_j)\vx_j)$. For simplicity we can write $L'$ as $L(\vw+\vz)$ for same $\vz \in \LL$.
 Remark, that for $j\notin I$ the space $\ExtX {L(\vx)}{L(\vw+(1+l_j)\vx_j)}\neq 0$, and there are no maps $\widetilde{x_j^{l_j+1}}$ for $j\notin I$.

First, we show that
$$(\star)\qquad \HomX E{L(\vw+\vz)}=0\quad \textnormal{for}\quad 0\leq\vz\leq \vx.$$
Indeed, let $\vz=\sum_{j=1}^ta_j\vx_j$ be an element of $\LL$
with $0\leq a_j\leq l_j$ for all $j$ and let $z=x_1^{a_1}x_2^{a_2}\cdots x_t^{a_t}$ be a morphism form $L(\vw)$ to $L(\vw+\vz)$. Applying the functor $\HomX E-$ to the sequence
$0\lra  L(\vw)\uplra{z}L(\vw+\vz)\lra S\lra 0$ we obtain that $\ExtX E{L(\vw)}=0$.
Next, applying the functor $\HomX-{L(\vw+\vz)}$ to $\eta$ we obtain a long  exact sequence
\begin{equation}\nonumber
\begin{split}
0\lra&  \HomX {E}{L(\vw+\vz)} \lra \podwzorem{\HomX {L(\vw)}{L(\vw+\vz)}}{\iso\ k} \lra\\
\lra& \podwzorem{\ExtX {L(\vx)}{L(\vw+\vz)} }{\iso\ k} \lra \podwzorem{\ExtX {E}{L(\vw+\vz)}}{=\ 0} \lra \cdots
\end{split}
\end{equation}
By comparing dimensions we get $\HomX E{L(\vw+\vz)}=0$.

Let $h:E\lra L(\vw+\vz)$ be a non-zero morphism  for some $z \in \LL$.
 Because $\eta$ is a non-split exact sequence, the map $h\circ i:L(\vw)\uplra{i} E \uplra{h} L(\vw+\vz)$ is a non-isomorphism.
 If $h\circ i$ is the zero map, then $h$ factors through by $\pi$ and we are done. Suppose  now
  that $h\circ i\neq 0$. Then $\vz\geq 0$, because $\HomX{L(\vw)}{L(\vw+\vz)}\neq 0$. Moreover, from property $(\star)$ we have   $\vz \nleq \vx$, and so $\vx-\vz\ngeq 0$. Now, we prove that there are maps $h_j:L((l_j+1)\vx_j +\vw)\lra L(\vw+\vz)$ for $j\in I$ such that $h\circ i=\sum_{j\in I}h_j\circ x_j^{l_j+1}$.
  Since $z\geq 0$ and $\vx-\vz\ngeq 0$,
   after standard calculation in the group $\Lp$, we see that there is an index $j_0\in I$ such that $\vz-(l_{j_0}+1)\vx_{j_0}\geq 0$.
    Therefore
    there is a map $h_{j_0}:L((l_{j_0}+1)\vx_{j_0} +\vw)\lra L(\vw+\vz)$ such that $h\circ i=h_{j_0}\circ x_{j_0}^{l_{j_0}+1}$.
    Further we define $h_j=0$ for $j \neq j_0$.

Then we have
$$\left(h-\sum_{j\in I}h_j\circ \widetilde{x_j^{l_j+1}}\right)\circ i=h\circ i-\sum_{j\in I} h_j\circ (\widetilde{x_j^{l_j+1}}\circ i)=h\circ i-\sum_{j\in I}h_j\circ x_j^{l_j+1}=0,$$
and we conclude that $h-\sum_{j\in I}h_j\circ \widetilde{x_j^{l_j+1}}\in\ker(-\circ i)=\im(-\circ \pi)$. Thus
there is a map $g:L(\vx)\lra L(\vw+\vz)$ such that $g\circ\pi=h-\sum_{j\in I}h_j\circ \widetilde{x_j^{l_j+1}}$ and hence $h=g\circ\pi+\sum_{j\in I}h_j\circ \widetilde{x_j^{l_j+1}}$.

The $\mathrm{Hom}_\XX-$orthogonality is easy to check.
The minimality for the map  $j_E:E\lra L(\vx)\oplus\bigoplus_{j\in I}L((l_j+1)\vx_j+\vw)$
follows then from the $\mathrm{Hom}_\XX-$ orthogonality
 of the line bundles $L(\vx)$ and $L((l_j+1)\vx_j+\vw)$ for $j\in I$.
\end{proof}

\begin{remark}
	In \cite{Kussin:Lenzing:Meltzer:2013adv} it was shown that in the case of  weight type $(2,a,b)$ the suspension functor $[1]$
in the stable vector bundle category $\svect{\XX}$   coincides with the shift functor by $\vx_1$.
% $\sigma(\vx_1)$.
 Therefore there is a short exact sequence
	$$0\lra \pc{E}(-\vx_1)\lra E\lra \pc E \lra 0$$ for each indecomposable bundle $E$. Hence in this case we have
$\rk\pc E=2\rk  E$.
From the theorem above  we see that in the case of $t$ weights with $t>3$ there is an indecomposable, not exceptional rank two bundle
such that $\rk \pc E>4=2\rk E$.

For example in the case $(2,2,2,3)$ the projective cover of an indecomposable bundle of
the data $(L,\vx_4)$ has rank $5$. Therefore in the case $(2,p_2,p_3,\cdots, p_t)$ and $t>3$
the suspension functor cannot  be realized by a shift with an element from $\LL$.
\end{remark}

%%%%%%%%%%%%%%%%%%%%%%%%%%%%%%%
%%%%%%%%%%%%%%%%%%%%%%%%%%%%%%%
%%%%%%%%%%%%%%%%%%%%%%%%%%%%%%%
%%%%%%%%%%%%%%%%%%%%%%%%%%%%%%%

In the same way as   Lemma \ref{lem:projektywny_skladnik_jest_modulem}
and Lemma \ref{lem:przedstawienie_wiazki_rozszerzen_na_4_sposoby}.
the  section \ref{sec:three_weights}, we can prove  the following two lemmata.

\begin{lemma}
	\label{lem:projektywny_skladnik_jest_modulem_t_wag}
	Let $E$ be a non-zero vector bundle over $\XX$ of a type $(p_1,p_2,\dots,p_t)$
 with a projective cover $\pc{E}$. If $\ExtX TE=0$ for the canonical bundle $T$,
  then there is a line bundle $L$ in the decomposition $\pc E$ into a direct sum
   of line bundles  such that $\ExtX TL=0$. \kwadracik
\end{lemma}

From Theorem \ref{prop:extension_bundle} each extension bundle can be given by a line bundle $L$ and an element $\vx=\sum_{i\in I}l_i\vx_i+\sum_{i\notin I}(p_i-1)\vx_i\in\LL$ for some $I\subset \{1,2,\dots, t\}$ with $\# I=3$.

\begin{lemma} \label{lem:przedstawienie_wiazki_rozszerzen_na_4_sposoby_t_wag}
	Let $\XX$ be a weighted projective line of  type $(p_1,p_2,\dots, p_t)$ and let
	$E$ be an extension bundle given by a pair $(L,\vx)$,
 where $\vx=\sum_{i\in I}l_i\vx_i+\sum_{i\notin I}(p_i-1)\vx_i$. Denote by $\widehat{L_i}$ the direct summand
 $L(\vx-(1+l_i)\vx_i)$ of $\pc{E}$.	
	\begin{enumerate}[label=\upshape(\roman*), leftmargin=*, widest=iii]
		\item There is an exact sequence
		$$0\lra \widehat{L_i}\lra E\lra L\big((1+l_j)x_j+\vw\big)\lra 0,$$
		where $L\big((1+l_j)x_j+\vw\big)$ is a direct summand od $\ih{E}$.\label{lem:it:1}
		\item The extension bundle $E$ can be determined by the following pairs
$$\Big(L\big(\vx-(1+l_i)\vx_i\big)(-\vw),\quad 2\vw+2(1+l_i)\vx_i-\vx \Big) \quad\textnormal{for}\quad i\in I.$$ \label{lem:it:2}
\item The element $2\vw+2(1+l_i)\vx_i-\vx$ in  normal form has exactly
$t-3$ coefficients equal to $p_i-1$ for $i\notin I$. \label{lem:it:3}
	\end{enumerate}
\end{lemma}

\begin{proof}
 The statements \ref{lem:it:1} and \ref{lem:it:2} are proved in the same way as in the case of three weights.
 The part \ref{lem:it:3} follows from  calculations in the group $\LL$ and is left to the reader.
\end{proof}

The next result is a generalization of Theorem \ref{thm:three:weights} and can be proved analogously
using Lemma 4.2 and Lemma 4.3.

%with the analogue proof.
\begin{theorem}
Let $\La=\La(\pp,\lala)$ be a canonical algebra with $t$ arms.
 Then each exceptional $\La-$module of rank two can be described
  by matrices having entries $0$, $ \la_i$ and $- \la_i$. \kwadracik
\end{theorem}

%%%%%%%%%%%%%%%%%%%%%%%%%%%%%%%
%%%%%%%%%%%%%%%%%%%%%%%%%%%%%%%
%%%%%%%%%%%%%%%%%%%%%%%%%%%%%%%
%%%%%%%%%%%%%%%%%%%%%%%%%%%%%%%

\section{Exceptional cokernels}

In this section we will deal with  cokernels of  maps of the form
\begin{equation}\label{eq:cokerlnel_map}
\left[x_i^{b_i}\right]_{i\in I}:\Oo(\vy)\lra \bigoplus_{i\in I}\Oo(\vy+b_{i}\vx_i),
\end{equation}
where $I\subset \{1,2,\dots,t\}$, $0<b_i<p_i-1$ for $i\in I$ and $\vy\in\LL_+$.
We will prove that such cokernels are exceptional modules.
 Moreover every exceptional module of rank two, can be obtain in this way.
 Finally, by the cokernel construction,
  we compute matrices  for each exceptional module of rank two.

\begin{lemma} \label{lem:exceptional_cokernel}
	Consider an exact sequence
	$$(\star)\quad 0\lra F\uplra{f} G\uplra{\pi} E\lra 0$$
	in $\coh\XX$ such that the following conditions are satisfied.
	\begin{enumerate}[label=\upshape{C\arabic*}.]
		\item $F$ is exceptional, \label{lem:con:1}
		\item $\HomX GF=0=\ExtX GF$, \label{lem:con:2}
		\item $\ExtX GG=0$,\label{lem:con:3}
		\item The map $-\circ f:\EndX G\lra \HomX FG$ is an isomorphism. \label{lem:con:4}
	\end{enumerate}
	Then the following properties holds:
	\begin{enumerate}[label=\upshape(\roman*), leftmargin=*, widest=iii]
		\item $\ExtX EF\iso k$ and $\HomX EG=0=\ExtX EG$, \label{lem:it:5}
		\item $E$ is exceptional, \label{lem:it:6}
		\item Up to an isomorphism $E$ does not depend on the map $f$. \label{lem:it:7}
	\end{enumerate}
\end{lemma}

\begin{proof}
	 \ref{lem:it:5}. Applying the functor
$\HomX-F$ to the exact sequence $(\star)$ we obtain a long exact sequence
	$$\cdots\lra\podwzorem{\HomX GF}{=\ 0} \lra \EndX F \lra \ExtX EF  \lra \podwzorem{\ExtX GF}{=\ 0}\lra\cdots$$
	Therefore $\ExtX EF\iso\EndX F\iso k$.
	
	Furthermore, applying the functor $\HomX -G$ to $(\star)$  we obtain a long exact sequence
	\begin{equation}\nonumber
	\begin{split}
	0\lra  \HomX EG\uplra{-\circ\pi} \EndX G\uplra{-\circ f}\HomX FG \uplra{\de}  \ExtX EG &\lra\\
	\lra  \podwzorem{\ExtX GG}{=\ 0}\lra \ExtX FG &\lra 0.
	\end{split}
	\end{equation}
	Because $-\circ f$ is an  isomorphism we get $\HomX EG=0$ and $\ExtX EG=0$
	
	 \ref{lem:it:6}. Applying the functor $\HomX E-$ to the exact sequence $(\star)$  we have a long exact sequence
	\begin{equation}\nonumber
	\begin{split}
	0 &\lra \HomX EF \lra \nadwzorem{\HomX EG}{=\ 0} \lra \EndX E \lra \\
	&\lra \ExtX EF \lra \podwzorem{\ExtX E{G}}{=\ 0} \lra \ExtX EE\lra 0.
	\end{split}
	\end{equation}
	We conclude that
	$\ExtX EE=0$, and using (i) also that  $\EndX E\iso  \ExtX EF \iso k.$
	
	 \ref{lem:it:7}.
 Suppose that we have exact sequences
	$$0\lra F\uplra{f} G\lra E\lra 0\quad \text{and}\quad 0\lra F\uplra{\widehat{f}} G\lra \widehat{E}\lra 0$$
	%satisfy conditions \ref{lem:con:1}-\ref{lem:con:4}.
 Then we have  $[E]=[G]-[F]=[\widehat{E}]$ in the Grothendieck group
  ${\mathrm K}_0(\XX)$.
	From  \ref{lem:it:6} we know that the sheaves $E$ and $\widehat {E}$ are exceptional.
We infer from \cite[Proposition 4.4.1]{Meltzer:2004} that $E\iso \widehat{E}$.
\end{proof}

 We will study the following cases of  maps satisfying conditions \ref{lem:con:1}-\ref{lem:con:4} of the previous proposition.
	\begin{itemize}
		\item[\textbf{a.}] A map $\big[x_1^{b_1},\dots,x_t^{b_t}\big]^T:\Oo\lra \bigoplus_{i=1}^t\Oo(b_i\vx_i)$, where $1\leq b_i\leq p_i-1$.
		
		\item[\textbf{b.}] For $J\subseteq \{1,\dots,t\}$ we consider
		$[x_i^{b_i}]_{i\in J}:\Oo\lra \bigoplus_{i\in J}\Oo(b_i\vx_i)$,
		where $1\leq b_i\leq p_i-1$ for $i=1,\dots,t$.
		
		\item[\textbf{c.}] If $f:F\lra G$ satisfies the conditions C1.-C4.
 then for each $\vx\in\LL$ the map $f(\vx):F(\vx)\lra G(\vx)$ also satisfies these conditions.
 %\ref{lem:con:1}-\ref{lem:con:4}.
	\end{itemize}
%\end{examples}

Note that if  $G$ is in $\modplus\La$ in the sequence $(\star)$, then $E$ is also in $\modplus\La$.
Therefore in the  cases \textbf{a.} and \textbf{b.} above the cokernels are  exceptional $\La-$modules.
\medskip

Let $\XX=\XX(p_1,p_2,\dots, p_t)$ be a weighted projective line
 and let $b_i$, $i=1,\dots,t$ are natural number such that
$1\leq b_i \leq p_i-1$. Denote $I=\{i : b_i<p_i-1\}$ and assume that $\#I=3$.

\begin{proposition}\label{lem:exceptional_cokernel_is_extension_bundle}
	The cokernel $E$ of the exact sequence
 $$(\star)\quad 0\lra L\uplra{f=\big[x_i^{b_i}\big]_{i\in I}}\bigoplus_{i\in I}L(b_i\vx_i)\uplra{\pi}E\lra 0$$
	is an extension bundle with data
	$$\Big(L(b_{i_0}x_{i_0}-\vw),\quad \vw+\sum_{i\in I}b_i\vx_i-2b_{i_0}\vx_{i_0} \Big)\quad \textnormal{for each}\quad i_0\in I.$$
	Moreover for each $i_0\in I$ the line bundle $L(b_{i_0}\vx_{i_0})$  is a direct summand of the projective cover $\pc{E}$.
\end{proposition}

\begin{proof}
	From Lemma \ref{lem:exceptional_cokernel} we conclude
	$$\HomX E{L(b_{i_0}\vx_{i_0})}=0=\ExtX E{L(b_{i_0}\vx_{i_0})}\quad\textnormal{for each}\quad {i_0}\in I.$$
	Applying the functor $\HomX {L(b_{i_0}\vx_{i_0})}-$ to $(\star)$ we obtain a long exact sequence
	\begin{equation}\nonumber
	\begin{split}
	0\lra& \nadwzorem{\HomX { L(b_{i_0}\vx_{i_0})} L}{=\ 0}\lra\HomX { L(b_{i_0}\vx_{i_0})}{\bigoplus_{i=1}^3 L(b_i\vx_i)}\lra \\
	\lra & \HomX { L(b_{i_0}\vx_{i_0})}E\lra  \podwzorem{\ExtX { L(b_{i_0}\vx_{i_0})} L}{=\ 0}\lra\cdots.
	\end{split}
	\end{equation}

	Therefore $\HomX {L(b_{i_0}\vx_{i_0})}E\iso\HomX {L(b_{i_0}\vx_{i_0})}{\bigoplus_{i=1}^3L(b_i\vx_i)}\iso k.$
	%%%%%%%%%%
	We denote $\pi=\big[\pi_i\big]_{i\in I}$.
	Then each sequence
	$$\eta_i:\quad 0\lra L(b_i\vx_i)\uplra{\pi_i} E\lra C_i\lra 0\quad \textnormal{for}\quad i\in I$$
	satisfies the conditions \ref{lem:con:1}-\ref{lem:con:3} from Lemma \ref{lem:exceptional_cokernel}.
 Because $\EndX E\iso k\iso\HomX {L(b_i\vx_i)}E$, the map
 $-\circ\pi_i:\EndX E\lra \HomX {L(b_i\vx_i)}E$ is an isomorphism
 or is zero. Further, since $\pi\neq 0$, at least one map  $-\circ\pi_i$ is an isomorphism.
	
	Assume that $-\circ\pi_{j_0}$ is an isomorphism. Then from Lemma \ref{lem:exceptional_cokernel},
the  term $C_{j_0}$ of $\eta_{j_0}$ is exceptional, therefore it is a line bundle.
 Moreover  $\det C_{j_0}=\det E-\det L(b_{j_0}\vx_{j_0})=\det L+\sum_{i\in I, i\neq {j_0}} b_i\vx_i$,
  and it follows that $C_{j_0}\iso L\Big(\sum_{i\in I, i\neq {j_0}} b_i\vx_i\Big)$.
	The exact sequence $\eta_{j_0}$ can be written in the following form
	$$0\lra L(b_{j_0}\vx_{j_0}-\vw)(\vw)\lra E\lra L(b_{j_0}\vx_{j_0}-\vw)(\vx)\lra 0,$$
	where
	\begin{equation}\nonumber
	\begin{split}
	\vx= &\vw-b_{j_0}\vx_{j_0}+\sum_{i\in I, i\neq {j_0}} b_i\vx_i \\
	=& (p_{j_0}-b_{j_0}-1)\vx_{j_0}+\sum_{i\in I, i\neq {j_0}} (b_i-1)\vx_i +\sum_{j\notin I}(p_j-1)\vx_j.
	\end{split}
	\end{equation}
	Then the element $\vx$ satisfies the inequality $0\leq \vx\leq\vdom$,
and consequently $E$ is an extension bundle with the data $\Big(L(b_{j_0}x_{j_0}-\vw),\vx\Big)$.
	
	Moreover, from Theorem \ref{thm:projective_covers_t-weight}, the line bundles  $L(b_i\vx_i)$
for $i\in I$, $i\neq j$ are direct summands of the projective cover of the vector bundle $E$. Thus, from Lemma \ref{lem:przedstawienie_wiazki_rozszerzen_na_4_sposoby_t_wag}
 the bundle $E$ is an extension bundle with  the data
  $\Big(L(b_{i_0}x_{i_0}-\vw),\vw+\sum_{i\in I}b_i\vx_i-2b_{i_0}\vx_{i_0}\Big)$ for each ${i_0}\in I$.
\end{proof}

%\begin{Cor}\label{cor:extension_bundles_are_coker_modules}
%	Each extension bundle $E=E_{L}\langle\vx\rangle$ for $\vx=\sum _{i\in I}l_i\vx_i+\sum_{j\notin I}(p_j-1)\vx_j$ is a cokernel of a map
%	$$\left[x_i^{b_i}\right]_{i\in I}:\widehat{L}_j\lra \bigoplus_{i=1}^3\widehat{L}_j(b_{i,j}\vx_i),\quad \text{for}\quad j\in I$$
%	where $$\widehat{L}_j=L(\vw-b_{j,j}\vx_j)\quad \text{and}\quad b_{i,j}=\left\{\begin{array}{ccc}
%	(p_j-l_j-1) &\text{for}& i=j\\
%	l_i+1 &\text{for}& i\neq j
%	\end{array}\right. .$$
%\end{Cor}

\begin{proposition}\label{lem:wiazka_rozszerzen_jako_kojadro}
	Let $E=\extb L\vx$ be an extension bundle with $\vx=\sum_{i\in I}l_i\vx_i+\sum_{j\notin I}(p_j-1)\vx_j$ and $\# I =3$.
	Then $E$ is the cokernel in the following exact sequence
	$$0\lra L(\vx-\vc)\uplra{\Big[x_i^{p_i-l_i-1}\Big]_{i\in I}} \bigoplus_{i\in I} L\Big(\vx-(1+l_i)\vx_i\Big)\lra E\lra 0.$$
	Moreover, the line bundles
 $L\Big(\vx-(1+l_i)\vx_i\Big)$ are direct summands of $\pc{E}$ and
 $L(\vx-\vc)$ is direct summand of $\ih{E}(-\vc)$.
\end{proposition}

\begin{proof}
	Consider the map $\Big[x_i^{p_i-l_i-1}\Big]_{i\in I}: L(\vx-\vc)\lra\bigoplus_{i\in I} L\Big(\vx-(1+l_i)\vx_i\Big)$. This map is a monomorphism and from the Proposition \ref{lem:exceptional_cokernel_is_extension_bundle} the cokernel of the map $\Big[x_i^{p_i-l_i-1}\Big]_{i\in I}$ is the extension bundle with data $(L,\vx)$. Hence $\mathrm{coker} \Big[x_i^{p_i-l_i-1}\Big]_{i\in I}\iso \extb{L}{\vx}$. The second claim follows from the form of the projective cover and the injective hull of the extension bundle $\extb{L}{\vx}.$
\end{proof}

\begin{lemma}\label{lem:wiazka_liniowa_jest_modulem}
	A line bundle $L$ is in $\modplus\La$ if and only if $\det L\geq 0$.
\end{lemma}

\begin{proof}
	Assume first  that $L$ is in $\modplus\La$, ie. $\ExtX TL=0$.
 Then $\HomX TL\neq 0$, so there is an element $\vx\in \LL$ such that $0\leq \vx \leq \vc$ and
 $\HomX {\Oo(\vx)}{L}\neq 0$. Therefore $\det L-\vx\geq 0$, so
	$\det L=\podwzorem{\det L-\vx}{\geq 0} + \podwzorem{\vx}{\geq 0} \geq 0.$
	
	Now, assume that $\det L=n\vc+\sum_{i=1}^ta_i\vx_i\geq 0$.
 Let  $\vx \in \LL$ satisfy that $0\leq \vx\leq \vc$. Then
	$\vx+\vw-\det L = \vx+\sum_{i=1}^t(p_i-a_i-1)\vx_i-(n+2)\vc\ngeq 0$.
	Therefore 	
	$\ExtX {\Oo(\vx)}L\cong D\HomX L{\Oo(\vx+\vw)}=0,$
	and consequently the line bundle $L$ is in $\modplus\La$.
\end{proof}

\begin{proposition}
	\label{cor:4_przedstawienia_E_jako_kojadro}
	Let $E$ be an extension bundle.
	\begin{enumerate}[label=\upshape(\roman*), leftmargin=*, widest=iii]
		\item For each direct summand $\widehat{L}$ of $\ih{E}$ there is a short exact sequence
		$$\eta_{\widehat{L}}:\quad 0\lra \widehat{L}(-\vc)\lra \bigoplus_{i=1}^3L_i\lra E\lra 0, $$
		where the $L_i$ are pairwise distinct direct summands of projective cover $\pc{E}$.\label{lem:it:4a}
		
		\item If $E$ is a $\La-$module from $\modplus\La$, then for at least one direct summand $\widehat{L}$ of $\ih{E}$
the line bundle $\widehat{L}(-\vc)$ is in $\modplus\La$ and $\eta_{\widehat{L}}$ is a sequence of $\La-$modules. \label{lem:it:4b}
	\end{enumerate}	
\end{proposition}

\begin{proof}
	The statement (i) is a consequence of Lemma \ref{lem:przedstawienie_wiazki_rozszerzen_na_4_sposoby_t_wag}
 and  Proposition \ref{lem:wiazka_rozszerzen_jako_kojadro}.

	 \ref{lem:it:4b}.
	Since $E$ is an extension bundle  there is an exact sequence
	$$0\lra L(\vw)\lra E\lra L(\vx)\lra 0,$$
	of $\La-$modules,
 where	$\vx=\sum_{i\in I}l_i\vx_i+\sum_{j\notin I}(p_j-1)\vx_j$ with $\# I =3$ and $0\leq l_i\leq p_i-2$.
  Recall form Lemma \ref{lem:wiazka_liniowa_jest_modulem} that $\det L(\vw)\geq 0$ and $\det L(\vx)\geq 0$.
		
		The direct summands of the injective hull $\ih{E}$ are as follows
	$$L(\vx),\quad L\Big(\vw+(1+l_i)\vx_i\Big) \quad \textnormal{for}\quad i\in I.$$

	If the line bundle $L(\vx-\vc)$ is in $\modplus\La$, we put $\widehat{L}=L(\vx)$ and the claim holds.
Assume now that $L(\vx-\vc)$ does  not belong to $\modplus\La$, then $\det L(\vx-\vc)\ngeq 0$.
	We write $\det L$ in normal form $\det L=n\vc+\sum_{i=1}^t a_i\vx_i$ and define two numbers $m$ and $m_I$ as follows
	$$m:=\#\big\{i\mid i\notin I \wedge a_i>0 \big\}\quad m_I:= \#\big\{i\mid i\in I \wedge a_i+l_i>p_i \big\}.$$
	Then we can write $\det L(\vx)$ in normal form
	 \begin{align*}
		\det L(\vx)&=n\vc +  \sum_{i=1}^t a_i\vx_i + \sum_{i\in I}l_i\vx_i+\sum_{i\notin I}(p_i-1)\vx_i\\
		&=n\vc +  \sum_{i\in I} (a_i+l_i)\vx_i  +\sum_{i\notin I}(p_i-1+l_i)\vx_i\\
		&=(n+m+m_I)\vc +  \sum_{i\in I} b_i\vx_i  +\sum_{i\notin I}c_i\vx_i \geq 0,
	\end{align*}
	where
	$$b_i=\left\{\begin{array}{ccc}
		a_i+l_i-p_i &\textnormal{if}&a_i+l_i\geq p_i\\
		a_i+l_i &\textnormal{if}&a_i+l_i<p_i
	\end{array}\right. \quad\textnormal{and} \quad
	c_i=\left\{\begin{array}{ccc}
	p_i-1 &\textnormal{if}&l_i=0\\
	l_i-1 &\textnormal{if}&l_i>0
	\end{array}\right. $$
Since
	\begin{align*}
		\det L(\vx-\vc)&=(n+m+m_I-1)\vc +  \sum_{i\in I} b_i\vx_i  +\sum_{i\notin I}c_i\vx_i \ngeq 0,
	\end{align*}
%$\det L(\vx-\vc)=(n+m+m_I-1)\vc +  \sum_{i\in I} b_i\vx_i  +\sum_{i\notin I}c_i\vx_i \ngeq 0,$
   we have $n+m+m_I=0$, hence $(\star)$ $n+m=-m_I$.
	Similarly we compute the determinant for the line bundle $L(\vw)$.	We have
	\begin{align*}
	\det L(\vw)&=n\vc +  \sum_{i=1}^t a_i\vx_i  +(t-2)\vc-\sum_{i=1}^t\vx_i = (n+t-2)\vc+ \sum_{i=1}^t (a_i-1)\vx_i,	
	\end{align*}
	where
   $\sum_{i=1}^t (a_i-1)\vx_i = -(t-3-m)\vc+\sum_{i\in I}(a_i-1)\vx_i + \sum_{i\notin I}b_i\vx_i$. We denote by $d_i$ the number $p_i-1$ if $a_i=0$ or $a_i-1$ if $a_i>0$. Then
   \begin{align*}
   \det L(\vw)&=(n+m+1)\vc +  \sum_{i\in I} (a_i-1)\vx_i  +\sum_{i\notin I}d_i\vx_i\\
   &=(1-m_I)\vc +  \sum_{i\in I} (a_i-1)\vx_i  +\sum_{i\notin I}d_i\vx_i \geq 0.
   \end{align*}
	Therefore $m_I$ is equal to $0$ or $1$. Moreover if $m_I=1$ then $a_i>0$ for all $i\in I$, and if $m_I=0$,
then at most one of the numbers $a_i$  for $i\in I$ is $0$.
	
	In the case $m_I=1$   there is an index $i_{0}\in I$ such that  $a_{i_0}+l_{i_0}\geq p_{i_0}$ Then
	\begin{align*}
		\det L\Big(\vw+(1+l_{i_0})\vx_{i_0}-\vc\Big)&= -\vc +  \sum_{i\in I} (a_i-1)\vx_i+ (1+l_{i_0})\vx_{i_0}  +\sum_{i\notin I}d_i\vx_i\\
		&=(a_{i_0}+l_{i_0}-p_{i_0})\vx_{i_0}+\sum_{i\in I, i\neq i_0 } \podwzorem{(a_i-1)}{\geq 0}\vx_i  +\sum_{i\notin I}d_i\vx_i  \geq 0.
	\end{align*}
	Therefore $L\Big(\vw+(1+l_{i_0})\vx_{i_0}-\vc\Big)$ is in $\modplus\La$ and we put $\widehat{L}=L\Big(\vw+(1+l_{i_0})\vx_{i_0}\Big).$

   In the case that $m_I=0$ if $a_i>0$ for all $i\in I$ then each line bundle
   $L\Big(\vw+(1+l_{i})\vx_{i}-\vc\Big)$ is in $\modplus\La$
    and each of those line bundles gives us the claim.
     If $a_{i_0}=0$ for some $i_0\in I$, then only
     $L\Big(\vw+(1+l_{i_0})\vx_{i_0}-\vc\Big)$ is in $\modplus\La$ and we put
     $\widehat{L}=L\Big(\vw+(1+l_{i_0})\vx_{i_0}\Big)$.
	
\end{proof}

As a conclusion of the previous proposition, we obtain an improvement of
Lemma \ref{lem:projektywny_skladnik_jest_modulem}.

\begin{corollary}
	If an extension bundle $E$ is a $\La-$module, then at least three direct summands of $\pc{E}$
 are also $\La-$modules. \kwadracik
\end{corollary}

Recall that we work with a map
%\eqref{eq:cokerlnel_map}
of the form
$$f^{b_{i_1},b_{i_2},b_{i_3}}_{\vy}=\left[x_i^{b_i}\right]_{i\in I}:\Oo(\vy)\lra \bigoplus_{i\in I}\Oo(\vy+b_{i}\vx_i),$$
where $I=\{i_1,i_2,i_3\mid i_1<i_2<i_3\}$ is a subset of $\{1,2,\dots,t\}$, $0<b_i<p_i-1$ for $i\in I$ and $\vy\in \LL_+$. If we write
the element $\vy$ in  normal form
$$\vy=n\vc+\sum_{i=1}^ta_i\vx_i,\quad n\in\ZZ_{\geq 0}\quad\textnormal{and}\quad 0\leq a_i\leq p_i-1 \quad\text{for}\quad i=1,2,\dots,t,$$
then we distinguish the following $8$ cases:
\begin{enumerate}[label=\textbf{A}]
	\item $a_j+b_j<p_j$ for all $j\in I$, \label{case:alpha}
\end{enumerate}
\begin{enumerate}[label=\textbf{B}$_\arabic*$]
	\item $a_{i_1}+b_{i_1}\geq p_{i_1}$ and $a_j+b_j<p_j$ for $j\in I-\{i_1\}$, \label{case:beta:1}
	\item $a_{i_2}+b_{i_2}\geq p_{i_2}$ and $a_j+b_j<p_j$ for $j\in I-\{i_2\}$, \label{case:beta:2}
	\item $a_{i_3}+b_{i_3}\geq p_{i_3}$ and $a_j+b_j<p_j$ for $j\in I-\{i_3\}$, \label{case:beta:3}
\end{enumerate}
\begin{enumerate}[label=\textbf{C}$_\arabic*$]
	\item $a_{i_1}+b_{i_1} < p_{i_1}$ and $a_j+b_j\geq p_j$ for $j\in I-\{i_1\}$,
	\label{case:gamma:1}
	\item $a_{i_2}+b_{i_2} < p_{i_2}$ and $a_j+b_j\geq p_j$ for $j\in I-\{i_2\}$,
	\label{case:gamma:2}
	\item $a_{i_3}+b_{i_3} < p_{i_3}$ and $a_j+b_j\geq p_j$ for $j\in I-\{i_3\}$,
	\label{case:gamma:3}
\end{enumerate}
\begin{enumerate}[label=\textbf{D}]
	\item $a_i+b_i\geq p_i$ for all $i\in I$. \label{case:delta}
\end{enumerate}
   In the following lemma we proof that it is sufficient to study
   the cases \ref{case:alpha} or \ref{case:beta:3}.

\begin{lemma}\label{lem:each_module_is_type_A_or_B}
	Each extension module can be obtained as a cokernel of the map $f^{b_{i_1},b_{i_2},b_{i_3}}_{\vy}$
in the case \ref{case:alpha} or \ref{case:beta:3}.
\end{lemma}

\begin{proof}
	We proof that the cokernels in the cases \ref{case:beta:1}, \ref{case:beta:2} and \ref{case:delta} are isomorphic to
 cokernels of the case \ref{case:beta:3}. Moreover the  cokernels in the cases
 \ref{case:gamma:1},  \ref{case:gamma:2} and  \ref{case:gamma:3}
  are isomorphic to cokernels of the case \ref{case:alpha}.
	
	Let $E$ be an extension module in the case \ref{case:beta:2}, thus $E$ is the cokernel of a map
	$$f^{b_{i_1},b_{i_2},b_{i_3}}_{\vy},\quad \textnormal{where}\quad a_{i_1}+b_{i_1}<p_{i_1},\quad a_{i_2}+b_{i_2}\geq p_{i_2}, \quad  a_{i_3}+b_{i_3}<p_{i_3}.$$
	 From  Lemma \ref{lem:exceptional_cokernel_is_extension_bundle}, applied to $i_1$, we infer
 that $E$ is an extension bundle $\extb L\vx$ with data $(L,\vx)$ such that
	\begin{equation}\nonumber
	\begin{split}
	\det L&=\vy+b_{i_3}\vx_{i_3}-\vw=n\vc+(a_{i_3}+b_{i_3})\vx_{i_3}+\sum_{j\neq i_3} a_j\vx_j-\vw,\\
	\vx&=\vw+b_{i_1}\vx_1+b_{i_2}\vx_{i_2}-b_{i_3}\vx_{i_3},
	\end{split}
	\end{equation}
	Consider the map
	$$f^{d_{i_1},d_{i_2},d_{i_3}}_{\vz}=\left[\begin{array}{c}
	x_{i_1}^{d_{i_1}}\\
	x_{i_2}^{d_{i_2}}\\
	x_{i_3}^{d_{i_3}}
	\end{array}\right]:\Oo(\vz)\lra \bigoplus_{i=1}^3\Oo(\vz+b_{i}\vx_i),$$
	where
	 \begin{align*}
		 	d_{i_1}&=b_{i_1},\quad d_{i_2}=p_{i_2}-b_{i_2},\quad d_{i_3}=p_{i_3}-b_{i_3}\\
		 \textnormal{and} &\\
		 \vz&=\vy+b_{i_2}\vx_{i_2}+b_{i_3}\vx_{i_3} -\vc =\\ &=n\vc+(a_{i_2}+b_{i_2}-p_{i_2}) \vx_{i_2} +(a_{i_3}+b_{i_3})\vx_{i_3}+\sum_{j\neq i_2,i_3}a_{j}\vx_{j}.
	\end{align*}
	Then from  Lemma \ref{lem:exceptional_cokernel_is_extension_bundle}, applied
 to $i_2$, we conclude that the cokernel of the map $f^{d_{i_1},d_{i_3},d_{i_3}}_{\vz}$ is an extension bundle with data $(\widehat{L}, \widehat\vx)$, such that
	\begin{equation}\nonumber
	\begin{split}
	\det \widehat{L}&=\vz+d_{i_2}\vx_{i_2}-\vw=n\vc+(a_{i_3}+b_{i_3})\vx_{i_3}+ \sum_{j\neq i_3} a_j\vx_j-\vw,,\\
	\widehat{\vx}&= \vw+d_{i_1}\vx_{i_1} -d_{i_2}\vx_{i_2}+d_{i_3}\vx_{i_3}=\vw+b_{i_1}\vx_{i_1}-(p_{i_2}-b_{i_2})\vx_{i_2}+(p_{i_3}-b_{i_3})\vx_{i_3}=\\ &=\vw +b_{i_1}\vx_{i_1}  +b_{i_2}\vx_{i_2}-b_{i_3}\vx_{i_3}.
	\end{split}
	\end{equation}
	Therefore, the cokernels of the maps $f^{b_{i_1},b_{i_3},b_{i_3}}_{\vy}$ and $f^{d_{i_1},d_{i_3},d_{i_3}}_{\vz}$ are isomorphic.
Furthermore we have that
	\begin{align*}
	&d_{i_1}+a_{i_1} =a_{i_1}+b_{i_1}< p_{i_1},& d_{i_2}+ (a_{i_2}+b_{i_2}-p_{i_2})=a_{i_2}< p_{i_2},\\  &d_{i_3}+(a_{i_3}+b_{i_3})=a_{i_3}+p_{i_3}\geq p_{i_3},	
	\end{align*}	
	hence $E$ is isomorphic to an extension module in the case of \ref{case:beta:3}.
	
	For the other cases, we use the same kind of  arguments.
We only put in a table the choice of $d_{i_1}$, $d_{i_2}$, $d_{i_3}$ and $\vz$.
% Therefore is sufficient to only put table of  choosing of $d_{i_1}$, $d_{i_2}$, $d_{i_3}$ and $\vz$ in the other cases.
For simplicity, in the case $a_i+b_i\geq p_i$, we denote $a_i+b_i-p_i$ by $c_i$.
	$$\begin{array}{c|c|c|c|c}
	\text{Case} & d_{i_1} & d_{i_2} & d_{i_3} & \vz \\
	\hline\hline
	\textnormal{\ref{case:beta:1}}& p_{i_1}-b_{i_1} & b_{i_2} &p_{i_3}-b_{i_3}&n\vc+ c_{i_1}\vx_{i_1}+(a_{i_3}+b_{i_3})\vx_{i_3}+\sum\limits_{j\neq i_1,i_3}a_j\vx_j \\
	\hline
	\textnormal{\ref{case:gamma:1}}&b_{i_1}&p_{i_2}-b_{i_2}&p_{i_3}-b_{i_3}&(n+1)\vc+c_{i_2}\vx_{i_2}+c_{i_3}\vx_{i_3}+\sum\limits_{j\neq i_2,i_3}a_j\vx_j\\
	\textnormal{\ref{case:gamma:2}}&p_{i_1}-b_{i_1}&b_{i_2}&p_{i_3}-b_{i_3}&(n+1)\vc+c_{i_1}\vx_{i_1}+c_{i_3}\vx_{i_3}+\sum\limits_{j\neq i_1,i_3}a_j\vx_j\\
	\textnormal{\ref{case:gamma:3}}&p_{i_1}-b_{i_1}&p_{i_2}-b_{i_2}&b_{i_3}&(n+1)\vc+c_{i_1}\vx_{i_1}+c_{i_2}\vx_{i_2}+\sum\limits_{j\neq i_1,i_2}a_j\vx_j\\
	\hline
	\textnormal{\ref{case:delta}}&p_{i_1}-b_{i_1}&b_{i_2}&p_{i_3}-b_{i_3}&(n+1)\vc+c_{i_1}\vx_{i_1}+c_{i_3}\vx_{i_3}+\sum\limits_{j\neq i_1,i_3}a_j\vx_j\\
	\end{array} $$
\end{proof}

\subsection{The cokernel construction}

We consider an exact sequence of $\La-$modules
$$0\lra L\uplra{f} M\uplra{g} N\lra 0,$$
and we assume that representations
$$L=\rep{L}{Q}{\vx}{\al},\quad  M=\rep{M}{Q}{\vx}{\al}$$
by vector spaces and matrices
and also the morphism $f=(f_{\vx})_{\vx\in Q_0}$ are known.
Here $Q_0$ and $Q_1$ denote the set of vertices and arrows of the quiver of the canonical algebra,
respectively.

We will construct a representation  $\rep{N}{Q}{\vx}{\al}$ for $N$.
 The vector space $N_{\vx}$ is the cokernel of the linear map $f_{\vx}$,
 and $g_{\vx}$ the reduction modulo $\im f_{\vx}$.
  Let $v_1$,\dots, $v_m$ be a basis of $M_{\vx}$ and let $\dim_k \im f_{\vx}=l$.
  We have that
$M_{\vx}=\im f_{\vx}\oplus kv_1\oplus\cdots\oplus kv_{m-l}$
and that the set
$v_1+\im f_{\vx}$,\dots, $v_{m-l}+\im f_{\vx}$ is a basis of the linear space $N_{\vx}=M_{\vx}/\im f_{\vx}$.
Moreover, for $j\in\{1,\dots,l\}$, we have
$v_{m-l+j}=f_{\vx}(w_j)+a_{1,j}v_1+\cdots+a_{m-l,j}v_{m-l}$
for some $a_{i,j}\in k$ i $w_j\in L_{\vx}$.
Then
\begin{align*}
v_i\upmapsto{g_{\vx}} & \modf{v_i}{\vx}, &&\text{for}\quad i=1,\dots,m-l\\
v_{m-l+j}\upmapsto{g_{\vx}} & \sum_{i=1}^{m-n} a_{i,j}\modf{v_i}{\vx},&&\text{for}\quad i=1,\dots,m-l.
\end{align*}
Therefore
$g_{\vx}=\begin{array}{|ccc|ccc|}
\hline
&&&&&\\
&I_{m-l}& &&A&\\
&&&&&\\
\hline
\end{array} \;,$
where $A=[a_{i,j}]\in M_{m-l, l}(k)$.

Next we will determine the maps $N_{\al}$ for $\al\in Q_1$.
Let $\al:\vx\lra \vy$ be an arrow of the quiver of the algebra $\La$. Then the following diagram
\begin{equation}\label{eq:diragram_przemienny_dla_alpha}
\xymatrix @+5pt{
	0 \ar[r]& L_{\vy}\ar[r]^{f_{\vy}} \ar[d]_{L_{\al}}& \ar[d]_{M_{\al}} M_{\vy}\ar[r]^{g_{\vy}}& N_{\vy}\ar[r] & 0\\
	0\ar[r]& L_{\vx}\ar[r]^{f_{\vx}}& M_{\vx}\ar[r]^{g_{\vx}}& N_{\vx}\ar[r] & 0\\
}
\end{equation}
can be uniquely completed to a commutative diagram by the map
 $N_{\al}: N_{\vy} \rightarrow  N_{\vx},  \\   v+\im f_{\vy}\longmapsto M_{\al}(v)+\im f_{\vx}$.
% The map $N_{\al}$ is well-define, because left square in above diagram is commutative.

It is easily checked that the maps $N_{\alpha} $ satisfy the canonical relations.

\subsection{Construction of modules of type \ref{case:alpha}} \label{sec:modules_type_A}
Let $I\subset \{1,2,\dots, t\}$ with $\#I=3$, and let $\underline{b}=(b_i)_{i\in I}$ with $1\leq b_i\leq p_i-1$.
For $I$ and $\underline{b}$ we consider an exact sequence of vector bundles
	$$0\lra L\uplra{\big[x_i^{b_i}\big]_{i\in I}}
	\bigoplus_{i\in I} L(b_{i}\vx_{i})\uplra{g} \mathrm{coker}\big[x_i^{b_i}\big]_{i\in I}=:E \lra 0,$$
	where $L$ is a line bundl, with $\det L=n\vc+\sum\limits_{i\in I}a_i\vx_i\geq 0$ and $a_i+b_i<p_i$ for $i\in I$.
	We denote by $I_n$  the identity matrix of size $n$ and by $X_{n+m\times n}$, $Y_{n+m\times n}$, $Z_{n}(\la)$ the
following matrices
\begin{center}
	$X_{n+m\times n}:={\small\left[\begin{array}{c}
		I_n\\ \hline 0
		\end{array}\right]}\in M_{n+m,n}(k)$,
	$Y_{n+m\times n}:={\small\left[\begin{array}{c}
		0\\ \hline I_n
		\end{array}\right]}\in M_{n+m,n}(k)$, \\
	$Z_{n}(\la):={\small\left[\begin{array}{ccccc}
			1&0&\cdots& 0 &0\\
			\la& 1& &0 &0\\
			&&\ddots& \ddots&\\
			0&0&\cdots &\lambda &1\\
		\end{array}\right]}\in M_{n}(k)$.
\end{center}

	The $\La$-module attached to the line bundle $L$ has the following shape:
 $$\xymatrix @C -5pt{
 	L:&k^{n+1}\ar[ldd]_{\id}&\ar[l]_{\id}\cdots& \ar[l]_{\id}k^{n+1}&&& \ar[lll]_{L_{\al^{(1)}_{a_1+1}}=X_{n+1\times n}} k^n& \cdots \ar[l]_{\id} &\ar[l]_{\id} k^n& \\
 	&k^{n+1}\ar[ld]_{\id}&\ar[l]_{\id}\cdots& \ar[l]_{\id}k^{n+1}&&& \ar[lll]_{L_{\al^{(2)}_{a_2+1}}=Y_{n+1\times n}} k^n& \cdots \ar[l]_{\id} &\ar[l]_{\id} k^n& \\
 	%%%%
 	k^{n+1}&\ar[l]\cdots&& &\cdots&&&  &\cdots& \ar[lu]_{\id}\ar[l]_{\id}\ar[ld]_{\id}  \ar[luu]_{\id}\ar[ldd]_{\id}k^n\\
 	%%%%%
 	&k^{n+1}\ar[ul]_{Z_{n+1}(\la_i)}&\ar[l]_{\id}\cdots& \ar[l]_{\id}k^{n+1}&&& \ar[lll]_{L_{\al^{(i)}_{a_i+1}}=Y_{n+1\times n}} k^n& \cdots \ar[l]_{\id} &\ar[l]_{\id} k^n& \\
 	&k^{n+1}\ar[luu]^{Z_{n+1}(\la_t)}&\ar[l]_{\id}\cdots& \ar[l]_{\id}k^{n+1}&&& \ar[lll]_{L_{\al^{(t)}_{a_3+1}}=X_{n+1\times n}} k^n& \cdots \ar[l]_{\id} &\ar[l]_{\id} k^n&
 }$$
 \noindent
where $\id$ is the identity map  (see \cite[Proposition 3.4]{Meltzer:2007}).
If $a_i=0$ for some $i\geq 3$, then $L_{\al^{(i)}_1}=L_{\al^{(i)}_{a_i+1}}=Z_{n+1}(\lambda_i)\cdot X_{n+1\times n}$.

  The modules $L(b_i\vx_i)$ for $i\in I$ have a  similar shape, with the difference that in the $i-$th arm, the jump of dimension is
   realized for the arrow $\al^{(i)}_{a_i+b_i+1}.$

 First we compute matrices of maps $x_i^{b_i}:L\lra L(b_i\vx_i)$ for each $i\in I$.
%After standard calculations we obtain that
The map $x_i^{b_i}$ has the following matrices:

$\mu_i\cdot\left(\begin{array}{cllllll}
&\cdots & & \upramka{a_1\vx_1}{I_{n+1}}&\upramka{(a_1+1)\vx_1}{I_n} & \cdots&\\ \\
&\cdots \upramka{a_i\vx_i}{I_{n+1}} &\upramka{(a_i+1)\vx_i}{X_{n+1\times n}}& \cdots & \upramka{(a_i+b_i)\vx_i}{X_{n+1\times n}} & \upramka{(a_i+b_i+1)\vx_i}{I_{n}}\cdots &\\ \\
\ramka{I_{n+1}}& & &\cdots & & & \ramka{I_n}\\ \\
&\cdots & & \upramka{a_j\vx_j}{I_{n+1}}&\upramka{(a_j+1)\vx_j}{I_n} & \cdots&\\
& & & \cdots& &&
\end{array}\right),$
for some $\mu_i\in k$, where $i\neq 2$ and $j\neq i$.
Here the captions above the frames mean the vertices of the quiver and the matrices in the frames the matrives for them.
 Moreover, in the case  $i=2$ we need to switch from $X_{n+1\times n}$ to $Y_{n+1\times n}$.
Therefore the map $\big[x_i^{b_i}\big]_{i\in I}$ depends on the three scalars $\mu_{i_1}$, $\mu_{i_2}$ and $\mu_{i_3}$. From  Lemma \ref{lem:exceptional_cokernel} $(iii)$, we can put $\mu_{i_1}=1=\mu_{i_2}$ and $\mu_{i_3}=-1$.

The second step is the computation of the map $g:\bigoplus_{i\in I}L(b_i\vx_i)\lra E$. For this purpose we will use the following lemma from linear algebra, where for simplicity, we will use notation $B_{b\times a}$ for a matrix $B\in M_{b\times a}(k)$.
\begin{lemma}
	Let $0\lra V\uplra{f} W\uplra{\overline{\phantom{g}}} \mathrm{coker}(f)\lra 0$ be an exact
 sequence of linear maps, where $\dim V=a$, $\dim W=a+b+c$ and $\overline{\phantom{g}}$ is the
 reduction modulo $\im(f)$.
	\begin{enumerate}
		\item If $f$ has a block matrix form $\begin{array}{|c|}
		\hline B_{b\times a}\\ \hline C_{c\times a}\\ \hline -I_a \\ \hline
		\end{array}$\, ,  then the reduction map $\overline{\phantom{g}}$ has a block matrix form
		$$\begin{array}{|c|c|c|}
		\hline I_b & 0 & B_{b\times a}\\ \hline 0 & I_c & C_{c\times a} \\ \hline
		\end{array} \, .$$
		
		\item If $f$ has a block matrix form $\begin{array}{|c|}
		\hline I_a\\ \hline B_{b\times a}\\ \hline -C_{c\times a} \\ \hline
		\end{array}$ \, , then the reduction map $\overline{\phantom{g}}$ has a block matrix form
		$$\begin{array}{|r|c|c|}
		\hline  -B_{b\times a} & I_b & 0\\ \hline C_{c\times a} & 0 & I_c \\ \hline
		\end{array} \, .$$
	\end{enumerate}
\end{lemma}
	
\begin{proof}
	 $\mathrm{(i)}$
	Let  $v_1$, $\dots$, $v_a$ be a basis of $V$, and let
 $w^1_1$, \dots, $w^1_b$, $w^2_1$, \dots, $w^2_c$,$w^3_1$, \dots, $w^3_a$ be a basis of $W$.
	Furthermore, we choose $\overline{w^1_1}$, \dots, $\overline{w^1_b}$, $\overline{w^2_1}$, \dots, $\overline{w^2_c}$ as a basis
 of $\mathrm{coker}(f)$. Then the equalities
	\begin{align*}
	\overline{w^3_i}=\overline{w^3_i+f(w^3_i) }=\overline{Bw^3_i+Cw^3_i},\quad \textnormal{for} \quad i=1,2,\dots, a
	\end{align*}
	implies the claim.
	
 $\mathrm{(ii)}$ We choose $v_1$, $\dots$, $v_a$ as a basis of $V$,
  and $w^1_1$, \dots, $w^1_a$, $w^2_1$, \dots, $w^2_b$,$w^3_1$, \dots, $w^3_c$ as a
   basis of $W$. Further, we choose  $\overline{w^2_1}$, \dots, $\overline{w^2_b}$,
    $\overline{w^3_1}$, \dots, $\overline{w^3_c}$ as a basis of $\mathrm{coker}(f)$. Then
	\begin{align*}
	\overline{w^1_i}=\overline{w^1_i-f(w^1_i) }=\overline{-Bw^1_i+Cw^1_i},\quad \textnormal{for} \quad i=1,2,\dots, a
	\end{align*}
	so the claim holds.
\end{proof}

By using the  lemma above we obtain a matrix representation of the map
 $g=\big(g_{\vx}\big):\bigoplus_{i\in I}L(b_i\vx_i)\lra E$. Recall that $I=\{i_1, i_2,i_3\}$ is given in ascending order. Then
\begin{align*}
&g_{\vec{0}}=g_{\vx_j}=\cdots=g_{a_j\vx_j}= \blockmatrix{I_{n+1}}{0}{I_{n+1}}{0}{I_{n+1}}{I_{n+1} }&& \textnormal{for}\quad j={1,2,\dots,t}\\
&g_{(a_j+1)\vx_j}=\cdots=g_{\vc}= \blockmatrix{I_n}{0}{I_n}{0}{I_n}{I_n}&& \textnormal{for}\quad j\notin I\\
&g_{(a_i+b_i+1)\vx_i}=\cdots=g_{\vc}= \blockmatrix{I_n}{0}{I_n}{0}{I_n}{I_n}&& \textnormal{for}\quad i\in I\\
&g_{(a_{i_1}+1)\vx_{i_1}}=\cdots=g_{(a_{i_1}+b_{i_1})\vx_{i_1}}= \blockmatrix{I_{n+1}}{0}{X_{n+1\times n}}{0}{I_n}{I_n}&& \textnormal{for}\quad i_1\neq 2\\
&g_{(a_j+1)\vx_j}=\cdots=g_{(a_j+b_j)\vx_j}= \blockmatrix{I_{n}}{0}{I_{n+1}}{0}{I_n}{X_{n+1\times n}}&& \textnormal{for}\quad i_2\neq 2\\
&g_{(a_{i_3}+1)\vx_{i_3}}=\cdots=g_{(a_{i_3}+b_{i_3})\vx_{i_3}}= \blockmatrix{-I_{n}}{I_{n}}{0}{X_{n+1\times n}}{0}{I_{n+1}}.&&
\end{align*}
Note that in the case $i_1=2$ or $i_2=2$ we need to switch from $X_{n+1\times n}$, to $Y_{n+1\times n}$ in the above block matrices.
From this we obtain
\begin{proposition}
	The module of type $A$ has the following dimensional vector
	$$\left(\begin{array}{cllllll}
	& & & \vdots& & &\\ \\
	&\cdots  \upramka{a_i\vx_i}{2n+2} & \upramka{(a_i+1)\vx_i}{2n+1}& \cdots &  \upramka{(a_i+b_i)\vx_i}{2n+1} &  \upramka{(a_i+b_i+1)\vx_i}{2n}\cdots &\\ \\
	 \upramka{0}{2n+2}& & &\cdots & & &  \upramka{\vc}{2n}\\ \\
	&\cdots & &  \upramka{a_j\vx_j}{2n+2}& \upramka{(a_j+1)\vx_j}{2n} & \cdots&\\
	& & & \vdots& &&
	\end{array}\right),$$
where $i\in I$ and $j\notin I$. Here the captions above the frames denote the vertices of the quiver and the numbers in the frames
the dimensions of the vector spaces for them.

\end{proposition}

Finally we compute the matrices of the the module $E$, by completing the following square:
$$(\star)\quad \xymatrix @+10pt{
	G_{\vy}\ar[r]^{G_{\al}=\id}\ar[d]_{g_{\vy}} &G_{\vx}\ar[d]^{g_{\vx}}\\
	E_{\vy}\ar @{-->}[r]^{E_{\al}} &E_{\vx},
}$$
to a commutative diagram, for each arrow $\alpha:\vx\lra \vy$. Since the  maps $G_\alpha$ are monomorphism,
 each of these squares can be complete only in  one way.
If for some arrow $\al:\vx\lra \vy$, the maps $L(b_i\vx_i)_{\al}$ are identities
 for each $i\in I$ and $g_{\vx}=g_{\vy}$, then $E_{\al}$ is the identity map.
  Therefore we  need only to determine matrices for the following arrows:
\begin{align*}
&\al^{(j)}_1& \textnormal{for}&\quad j=3,4,\dots, t\\
&\al^{(j)}_{a_j+1} & \textnormal{for}&\quad j=1,2,\dots, t\\
&\al^{(i)}_{a_i+b_i+1}& \textnormal{for}&\quad i\in I
\end{align*}
In the case of the arrows of the $j-$th arm, with $j\neq i_3$ (and arrow $\al_{1}^{(i_3)}$)
we  deal with the following commutative diagram
$$\xymatrix @+18pt{
	G_{\vy}\ar[rrr]^{G_{\al}=\BlockMatrixDiagonalThree{D_{ b_1\times b_2}}{D_{c_1\times c_2}}{D_{a_1\times a_2}} } \ar[d]_{g_{\vy}=\BlockMatrixDwoThree{I_{b_2}}0{A_{b_2\times a_2}}0{I_{c_2}}{A_{c_2 \times a_2}} } &&&G_{\vx}\ar[d]^{g_{\vx}=\BlockMatrixDwoThree{I_{b_1}}0{B_{b_1 \times a_1 }}0{I_{c_1}}{B_{c_1 \times a_1}} }\\
	E_{\vy}\ar @{-->}[rrr]_{E_{\al}=\BlockMatrixDwoDwo{E_{b_1\times b_2 }}{E_{ b_1\times c_2}}{E_{c_1\times b_2}}{E_{c_1\times c_2}}} &&&E_{\vx}
} .$$
Then from the commutativity of the diagram above we get
$E_{\al}=\BlockMatrixDwoDwo{D_{ b_1\times b_2}}{0}{0}{D_{c_1\times c_2}}$.

In the case of arrow $\al_{a_{i_3}+1}^{(i_3)}$ we  deal with a commutative diagram of the form
$$\xymatrix @+18pt{
	G_{\vy}\ar[rrr]^{G_{\al}=\BlockMatrixDiagonalThree{D_{ b_1\times b_2}}{D_{c_1\times c_2}}{D_{a_1\times a_2}} } \ar[d]_{g_{\vy}=\BlockMatrixDwoThree{-A_{c_2\times b_2}}{I_{c_2}}0{A_{a_2 \times b_2}}0{I_{a_2}} } &&&G_{\vx}\ar[d]^{g_{\vx}=\BlockMatrixDwoThree{I_{b_1}}0{B_{b_1 \times a_1 }}0{I_{c_1}}{B_{c_1 \times a_1}} }\\
	E_{\vy}\ar @{-->}[rrr]_{E_{\al}=\BlockMatrixDwoDwo{E_{b_1\times c_2 }}{E_{ b_1\times a_2}}{E_{c_1\times c_2}}{E_{c_1\times a_2}}} &&&E_{\vx}
}.$$
Then
$E_{\al}=\BlockMatrixDwoDwo{0}{B_{b_1\times a_1}\cdot D_{ a_1\times a_2}}{D_{c_1\times c_2}}{B_{c_1\times a_1}\cdot D_{ a_1\times a_2} }$.

In the case of the arrow $\al_{a_{i_3}+b_{i_3}+1}^{(i_3)}$ we  deal with a commutative diagram of the form
$$\xymatrix @+18pt{
	G_{\vy}\ar[rrr]^{G_{\al}=\BlockMatrixDiagonalThree{D_{ b_1\times b_2}}{D_{c_1\times c_2}}{D_{a_1\times a_2}} } \ar[d]_{g_{\vy}=\BlockMatrixDwoThree{I_{b_2}}0{A_{b_2\times a_2}}0{I_{c_2}}{A_{c_2 \times a_2}} } &&&G_{\vx}\ar[d]^{ g_{\vx}=\BlockMatrixDwoThree{-B_{c_1 \times b_1 }}{I_{c_1}}0{B_{a_1\times b_1}}0{I_{a_1}} }\\
	E_{\vy}\ar @{-->}[rrr]_{E_{\al}=\BlockMatrixDwoDwo{E_{c_1\times b_2 }}{E_{ c_1\times c_2}}{E_{a_1\times b_2}}{E_{a_1\times c_2}}} &&&E_{\vx}
}.$$
Then
$E_{\al}=\BlockMatrixDwoDwo{-B_{c_1\times b_1}\cdot D_{ b_1\times b_2}}{ D_{c_1\times c_2}}{B_{a_1\times b_1}\cdot D_{ b_1\times b_2}}{0}$.

\begin{theorem}\label{lem:representation_of_type_A}
The extension module of Type A can be established by the following vector spaces and matrices. \\
{\small
	$\xymatrix @C -7pt{
	&\ar[ldd]\cdots&& &\cdots&& & && \cdots& \\
	&\cdots\ar[ld]|{E_{\al^{(i)}_{1}}}&k^{2n+2}\ar[l]_{\id}&& \ar[ll]_{E_{\al^{(i)}_{a_i+1} }}k^{2n+1}& \ar[l]_{\id}\cdots&\ar[l]_{\id}k^{2n+1} && k^{2n} \ar[ll]_{E_{\al^{(i)}_{a_i+b_i+1} }} &\ar[l]_{\id}\cdots &\\
	%%%%
	k^{2n+2}&&\ar[ll]|{E_{\al^{(j)}_{1}}}\cdots&& \ar[ll]k^{2n+2}&& k^{2n}\ar[ll]_{E_{\al^{(j)}_{a_j+1}}}    &&  \cdots\ar[ll]&& \ar[ll]\ar[lu]_{\id} \ar[ld]_{\id}  \ar[luu]_{\id}k^{2n}\\
	%%%%%
	&\ar[lu]^{E_{\al_{1}^{(t)}}}\cdots&& &\cdots&& & && \cdots&
}$  }
where
\begin{align*}
&E_{\al^{(1)}_{1}}=E_{\al^{(2)}_{1}}=\id,&& E_{\al^{(j)}_{1}}=\BlockMatrixDwoDwo{Z_{n+1}(\la_j)}{0}{0}{Z_{n+1}(\la_j)}\quad \textnormal{for}\quad j>2\\
&E_{\al^{(i_1)}_{a_{i_1}+1}}= E_{\al^{(i_2)}_{a_{i_2}+1}}=\BlockMatrixDwoDwo{I_{n+1}}{0}{0}{X_{n+1\times n}},&&
E_{\al^{(i_1)}_{a_{i_1}+b_{i_1}+1}}=E_{\al^{(i_2)}_{a_{i_2}+b_{i_2}+1}}=\BlockMatrixDwoDwo{X_{n+1\times n}}{0}{0}{I_{n}},\\
&E_{\al^{(i_3)}_{a_{i_3}+1}}= \BlockMatrixDwoDwo{0}{I_{n+1}}{X_{n+1\times n}}{I_{n+1}},&&
E_{\al^{(i_3)}_{a_{i_3}+b_{i_1}+1}}=\BlockMatrixDwoDwo{-I_n}{I_n}{X_{n+1\times n}}{0},\\
&E_{\al^{(j)}_{a_j+1} }=\BlockMatrixDwoDwo{X_{n+1\times n}}{0}{0}{X_{n+1\times n}} &&\textnormal{for}\quad j\notin \{i_1,i_2,i_3\}.
\end{align*}
In the case of the second  arm we need to switch from the matrices $X_{*\times*}$ to $Y_{*\times*}$.
 Moreover, if $a_i=0$, then the arrow $\alpha_1^{(i)}$ coincides with the arrow $\alpha_{a_i+1}^{(i)}$,
 and in this case
   in the place of $\alpha_1^{(i)}$ we put the composition of the above matrices $E_{\al^{(i)}_{1}}$ and $E_{\al^{(i)}_{a_{i}+1}}$.
\end{theorem}

\subsection{The modules of the type \ref{case:beta:3}}  This case is similar to the
previous computation. We will point out the differences using the same notations as before. In the case of type $B_3$ we assume that
$$a_{i_3}+b_{i_3}\geq p_{i_3}\quad \textnormal{and}\quad a_{i}+b_{i}< p_{i}\quad \textnormal{for}\quad i=i_1,i_2.$$
Let $c_{i_3}:=a_{i_3}+b_{i_3}-p_{i_3}$, then $0\leq c_{i_3}<a_{i_3}$.

The maps $x_{i}^{b_{i}}:\Oo(\vy)\lra \Oo(\vy+b_{i}\vx_{i})$ for $i=i_1$ or $i=i_2$ are the same as before. The map $x_{i_3}^{b_{i_3}}:\Oo(\vy)\lra \Oo(\vy+b_{i_3}\vx_{i_3})$ has the form $\mu_{i_3}\cdot(h_\vx)_{0\leq \vx\leq \vc}$, where for $j\neq i_3$ we have
\begin{align*}
&h_{0}=h_{\vx_j}=\cdots=h_{a_j\vx_j}=Z_{n+2\times n+1}(-\la_{i_3})\\
&h_{(a_j+1)\vx_j}=\cdots=h_{\vc}=Z_{n+1\times n}(-\la_{i_3})
\end{align*}
and for $i_3$ holds
\begin{align*}
&h_{0}=h_{\vx_{i_3}}=\cdots=h_{c_{i_3}\vx_{i_3}}=Z_{n+2\times n+1}(-\la_{i_3})\\
&h_{(c_{i_3}+1)\vx_i}=\cdots=h_{a_{i_3}\vx_{i_3}}=Z_{n+1}(-\la_{i_3})\\
&h_{(a_{i_3}+1)\vx_i}=\cdots=h_{\vc}=Z_{n+1\times n}(-\la_{i_3})
\end{align*}

Then the map $g=(g_\vx)_{0\leq \vx\leq \vc}:G\lra E$ has the following shape:
\begin{align*}
&g_{\vec{0}}=g_{\vx_j}=\cdots=g_{a_j\vx_j}= \blockmatrix{-I_{n+1}}{I_{n+1}}{0}{Z_{n+2\times n+1}(-\la_{i_3})}{0}{I_{n+2}}&& \textnormal{for}\quad j\neq i_3\\
%%%%%%%%
%%%%%%%%
%%%%%%%%
&g_{(a_j+1)\vx_j}=\cdots=g_{\vc}= \blockmatrix{-I_{n}}{I_{n}}{0}{Z_{n+1\times n}(-\la_{i_3})}{0}{I_{n+1}}&& \textnormal{for}\quad j\notin I\\
%%%%%%%%
%%%%%%%%
%%%%%%%%
&g_{(a_{i_1}+1)\vx_{i_1}}=\cdots=g_{(a_{i_1}+b_{i_1})\vx_{i_1}}= \blockmatrix{I_{n+1}}{-X_{n+1\times n}}{0}{0}{Z_{n+1\times n}(-\la_{i_3})}{I_{n+1}}&& \textnormal{for}\quad i_1\neq 2
\\
%%%%%%%%
%%%%%%%%
%%%%%%%%
&g_{(a_i+b_i+1)\vx_i}=\cdots=g_{\vc}= \blockmatrix{-I_{n}}{I_{n}}{0}{Z_{n+1\times n}(-\la_{i_3})}{0}{I_{n+1}}&& \textnormal{for}\quad i\in\{i_1,i_2\}\\
%%%%%%%%
%%%%%%%%
%%%%%%%%
&g_{(a_{i_2}+1)\vx_{i_2}}=\cdots=g_{(a_{i_2}+b_{i_2})\vx_{i_2}}= \blockmatrix{-X_{n+1\times n}}{I_{n+1}}{0}{Z_{n+1\times n}(-\la_{i_3})}{0}{I_{n+1}}&& \textnormal{for}\quad i_2\neq 2\\
%%%%%%%%
%%%%%%%%
%%%%%%%%
&g_{\vec{0}}=g_{\vx_{i_3}}=\cdots=g_{c_{i_3}\vx_{i_3}}= \blockmatrix{-I_{n+1}}{I_{n+1}}{0}{Z_{n+2\times n+1}(-\la_{i_3})}{0}{I_{n+2}}&& \\
%%%%%%%%
%%%%%%%%
%%%%%%%%
&g_{(c_{i_3}+1)\vx_{i_3}}=\cdots=g_{a_{i_3}\vx_{i_3}}= \blockmatrix{-I_{n+1}}{I_{n+1}}{0}{Z_{n+1}(-\la_{i_3})}{0}{I_{n+1}}.&& \\
%%%%%%%%
%%%%%%%%
%%%%%%%%
&g_{(a_{i_3}+1)\vx_{i_3}}=\cdots=g_{\vc}= \blockmatrix{-I_{n}}{I_{n}}{0}{Z_{n+1\times n}(-\la_{i_3})}{0}{I_{n+1}}.&&
\end{align*}
Lets remark, that if $i_k=2$ for same $k=1,2$ or $3$, then we need to switch matrices $X_{*,*}$ to $Y_{*,*}$ in $i_k$-arm.

\begin{proposition}
The module of type \ref{case:beta:3} has the following dimension vector.
$$\left(\begin{array}{ccccc}
&&{\LARGE{\cdots}}&& \\
&\upramka{\vx_i}{2n+3}\cdots\upramka{a_i\vx_i}{2n+3}&\upramka{(a_i+1)\vx_i}{2n+3}\cdots\upramka{(a_i+b_i)\vx_i}{2n+2}&\upramka{(a_i+b_i+1)\vx_i}{2n+2}\cdots\upramka{(p_i-1)\vx_i}{2n+1}& \\
&&{\Large{\cdots}}&& \\
\upramka{0}{2n+3}&\cdots&\upramka{a_j\vx_j}{2n+3}\quad\upramka{(a_j+1)\vx_j}{2n+1} &\cdots& \upramka{\vc}{2n+1}\\
&&{\Large{\cdots}}&& \\
&\upramka{\vx_{i_3}}{2n+3}\cdots\upramka{c_{i_3}\vx_{i_3}}{2n+3}&\upramka{(c_{i_3}+1)\vx_{i_3}}{2n+3}\cdots\upramka{(a_{i_3})\vx_{i_3}}{2n+2}&\upramka{(a_{i_3}+1)\vx_{i_3}}{2n+2}\cdots\upramka{(p_{i_3}-1)\vx_{i_3}}{2n+1}& \\
&&{\Large{\cdots}}&& \\
\end{array}\right),$$
where $i\in \{i_1,i_2\}$ and $j\notin I$.
Here the captions above the frames denote the vertices of the quiver and the numbers in the frames
the dimensions of the vector spaces for them.
\end{proposition}

Proceeding as in the case before we get representations for a module of type \ref{case:beta:3}.
\begin{theorem}
	\label{lem:representation_of_B_3}
	The extension module $E$ of type \ref{case:beta:3} can be exhibited by the following vector spaces and matrices.\\
	{\small		$
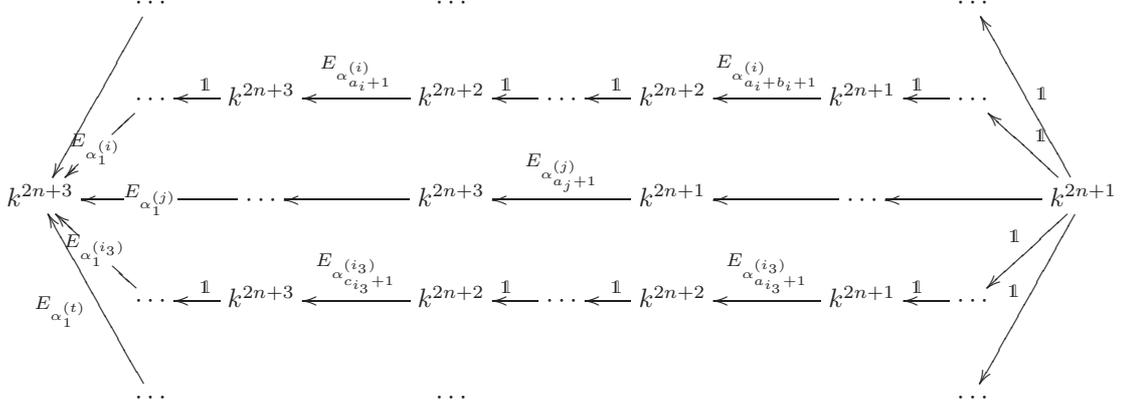
\xymatrix @C -7pt{
			&\ar[ldd]\cdots&& &\cdots&& & && \cdots& \\
			&\cdots\ar[ld]|{E_{\al^{(i)}_{1}}}&k^{2n+3}\ar[l]_{\id}&& \ar[ll]_{E_{\al^{(i)}_{a_i+1} }}k^{2n+2}& \ar[l]_{\id}\cdots&\ar[l]_{\id}k^{2n+2} && k^{2n+1} \ar[ll]_{E_{\al^{(i)}_{a_i+b_i+1} }} &\ar[l]_{\id}\cdots &\\
			%%%%
			k^{2n+3}&&\ar[ll]|{E_{\al^{(j)}_{1}}}\cdots&& \ar[ll]k^{2n+3}&& k^{2n+1}\ar[ll]_{E_{\al^{(j)}_{a_j+1}}}    &&  \cdots\ar[ll]&& \ar[ll]\ar[lu]_{\id} \ar[ld]_{\id}  \ar[luu]_{\id}\ar[ldd]_{\id}k^{2n+1}\\
			%%%%%
			&\cdots\ar[lu]|{E_{\al^{(i_3)}_{1}}}&k^{2n+3}\ar[l]_{\id}&& \ar[ll]_{E_{\al^{(i_3)}_{c_{i_3}+1} }}k^{2n+2}& \ar[l]_{\id}\cdots&\ar[l]_{\id}k^{2n+2} && k^{2n+1} \ar[ll]_{E_{\al^{(i_3)}_{a_{i_3}+1} }} &\ar[l]_{\id}\cdots &\\
			%%%
			&\ar[luu]^{E_{\al_{1}^{(t)}}}\cdots&& &\cdots&& & && \cdots&
		}$  }
	for $j\notin \{i_1,i_2,i_3\}$, where
	\begin{align*}
	&E_{\al^{(j)}_{1}}=\BlockMatrixDwoDwo{Z_{n+1}(\la_j)}{0}{0}{Z_{n+2}(\la_j)}\ \textnormal{for}\ j>2&&E_{\al^{(1)}_{1}}=E_{\al^{(2)}_{1}}=\id, \\
	%%%%%%%%%%%%%%%%%%%% i_1 arm
	&E_{\al^{(i_1)}_{a_{i_1}+1}}=\BlockMatrixDwoDwo{-I_{n+1}}{0}{Z_{n+2\times n+1}(-\la_{i_3})}{X_{n+2\times n+1}},&&
	E_{\al^{(i_1)}_{a_{i_1}+b_{i_1}+1}}=\BlockMatrixDwoDwo{-X_{n+1\times n}}{0}{Z_{n+1\times n}(-\la_{i_3})}{I_{n+1}},\\
	%%%%%%%%%%%%%%%%%%%% i_2 arm
	&E_{\al^{(i_2)}_{a_{i_2}+1}}=\BlockMatrixDwoDwo{I_{n+1}}{0}{0}{X_{n+2\times n+1}},&&
	E_{\al^{(i_2)}_{a_{i_2}+b_{i_2}+1}}=\BlockMatrixDwoDwo{X_{n+1\times n}}{0}{0}{I_{n+1}},\\
	%%%%%%%%%%%%%%%%%%%%% i_3 arm
	&E_{\al^{(i_3)}_{c_{i_3}+1}}= \BlockMatrixDwoDwo{I_{n+1}}{0}{0}{X_{n+2\times n+1}},&&
	E_{\al^{(i_3)}_{a_{i_3}+1}}=\BlockMatrixDwoDwo{X_{n+1\times n}}{0}{0}{I_{n+1}},\\
	%%%%%%%%%%%%%%%%% other arms
	&E_{\al^{(j)}_{a_j+1} }=\BlockMatrixDwoDwo{X_{n+1\times n}}{0}{0}{X_{n+2\times n+1}}&&\textnormal{for}\quad j\notin \{i_1,i_2,i_3\}.
	\end{align*}
In the case of the second arm we need to switch from the matrices $X_{*\times*}$ to $Y_{*\times*}$.
Moreover, if $a_i=0$, then the arrow $\alpha_1^{(i)}$ coincide with the arrow $\alpha_{a_i+1}^{(i)}$,
in this case in the place of $\alpha_1^{(i)}$ we put composition of  matrices $E_{\al^{(i)}_{1}}$ and $E_{\al^{(i)}_{a_{i}+1}}$.
\end{theorem}

\section{Exceptional modules  of the higher rank}

In this paper we have focused on the case of $\La$-modules of rank two.
We remark that the presented construction  of them by cokernels
can be applied also for exceptional modules of higher rank.
For this we have to consider  exact sequences of the form
$$0\lra L\uplra{f} \oplus_{j\in J} L(b_i\vx_i)\lra E\lra 0,$$
for $J\subset \{1,2,\dots,t\}$.
 In this case the cokernel $E$ is exceptional of rank $|J|-1$.
Therefore here we obtain exceptional modules of rank from $2$, to $t-1$.
It is an open question whether in this way we get
all exceptional modules of rank $r$, for $3 \leq r \leq t$.

\bibliographystyle{amsplain}

%\bibliography{biblio}

\begin{thebibliography}{17}

\bibitem{Dowbor:Meltzer:Mroz:2010}
P. Dowbor, H. Meltzer, and A. Mróz,
\emph{An algorithm for the construction of exceptional modules over tubular canonical algebras},
J. Algebra {\bf 323} (2010), no. 10, 2710--2734.

\bibitem{Dowbor:Meltzer:Mroz:2014slope}
P. Dowbor, H. Meltzer, and A. Mróz,
\emph{Parametrizations for integral slope homogeneous modules over tubular canonical algebras},
Algebr. Represent. Theory {\bf 17}(1) (2014), 321--356.

\bibitem{Dowbor:Meltzer:Mroz:2014bimodules}
P. Dowbor, H. Meltzer, and A. Mróz,
\emph{An algorithm for the construction of parametrizing bimodules for homogeneous modules over tubular canonical algebras},
Algebr. Represent. Theory {\bf 17}(1) (2014), 357--405.

\bibitem{Drozd:Kirichenko:1994} J. A. Drozd and V. V. Kirichenko,
 \emph{Finite Dimensional Algebras}, Springer-Verlag, Berlin, Heidelberg, New York, (1994).

\bibitem{Geigle:Lenzing:1987}
 W. Geigle and H. Lenzing,
  \emph{A class of weighted projective curves arising in representation
  theory of finite dimensional algebras}.
  In:  Singularities, representations of algebras, and vector
  bundles. Springer Lecture Notes in Mathematics {\bf 1273} (1987), 265--297.


\bibitem{Kedzierski:Meltzer}
D. E. Kędzierski and H. Meltzer,
\emph{Schofield induction for sheaves on weighted projective lines}.
Commun. Algebra 41, No. 6 (2013), 2033-2039.

\bibitem{Kedzierski:Meltzer:2020}
D. E. Kędzierski and H. Meltzer,
\emph{Exceptional modules over wild canonical algebras}, Colloquium Mathematicum {\bf 162} (2020), 159--180.

\bibitem{Komoda:Meltzer:2008}
S. Komoda and H. Meltzer,
\emph{Indecomposable modules for domestic canonical algebras
in arbitrary characteristic}, Int. J. Algebra {\bf 2} (2008), no. 4, 153--161.

\bibitem{Kussin:Lenzing:Meltzer:2013adv}
D. Kussin, H. Lenzing, and H. Meltzer,
\emph{Triangle singularities, ADE-chains, and
weighted projective lines}, Adv. Math. {\bf 237} (2013), 194--251.

\bibitem{Kussin:Lenzing:Meltzer:2013nil}
D. Kussin, H. Lenzing, and H. Meltzer,
\emph{Nilpotent operators and weighted projective lines}, Journal für die reine und angewandte Mathematik {\bf 685} (2013), 33--71.

\bibitem{Kussin:Meltzer:2007}
D. Kussin and H. Meltzer,
\emph{Indecomposable modules for domestic canonical algebras}.
J. Pure Appl. Algebra {\bf 211}, No. 2 (2007), 471--483.

\bibitem{Meltzer:2004}
H. Meltzer,
 \emph{Exceptional vector bundles, tilting sheaves and tilting complexes for weighted projective lines}.
 Memoirs AMS  {\bf 808} (2004).

\bibitem{Meltzer:2007}
H. Meltzer,
 \emph{Exceptional modules for tubular canonical algebras}. Algebr.
Represent. Theory {\bf 10}, No. 5 (2007), 481--496.


\bibitem{Ringel:1984}
C. M. Ringel,
  \emph{Tame algebras and integral quadratic forms}.
  Springer Lecture Notes in Mathematics {\bf 1099} (1984).

 \bibitem{Ringel:1998}
C. M. Ringel,
  \emph{Exceptional modules are tree modules}.
  Linear Algebra Appl. {\bf 275--276} (1998), 471--493.

\end{thebibliography}

\end{document}